\documentclass[oneside]{amsart}

\usepackage[a4paper]{geometry}

\usepackage{amssymb,amsfonts,amsmath,stmaryrd,todonotes}
\usepackage{comment}
\usepackage[all,arc]{xy}
\usepackage{enumerate}
\usepackage{mathrsfs}
\usepackage{todonotes}

\usepackage{tikz}
\usepackage{tikz-cd}
\usetikzlibrary{trees}
\usetikzlibrary[shapes]
\usetikzlibrary[arrows]
\usetikzlibrary{patterns}
\usetikzlibrary{fadings}
\usetikzlibrary{backgrounds}
\usetikzlibrary{decorations.pathreplacing}
\usetikzlibrary{decorations.pathmorphing}
\usetikzlibrary{positioning}

\def\biblio{\bibliography{duality}\bibliographystyle{alpha}}

\usepackage{xcolor} 
\usepackage{graphicx}
\usepackage[pagebackref]{hyperref}

\definecolor{dark-red}{rgb}{0.5,0.15,0.15}
\definecolor{dark-blue}{rgb}{0.15,0.15,0.6}
\definecolor{dark-green}{rgb}{0.15,0.6,0.15}
\hypersetup{
    colorlinks, linkcolor=dark-red,
    citecolor=dark-blue, urlcolor=dark-green
}
\newcommand{\iHom}{\underline{\operatorname{Hom}}}

\renewcommand*{\backref}[1]{}
\renewcommand*{\backrefalt}[4]{%
  \ifcase #1 %
No citations.
  \or
(cit. on p. #2).%
  \else
(cit on pp. #2).%
  \fi%
}
\usepackage{subfiles}

\usepackage[nameinlink,capitalise,noabbrev]{cleveref}
\newtheorem{thm}{Theorem}[section]
\newtheorem*{thm*}{Theorem}
\newtheorem{cor}[thm]{Corollary}
\newtheorem*{cor*}{Corollary}
\newtheorem{prop}[thm]{Proposition}
\newtheorem{lem}[thm]{Lemma}
\newtheorem{conj}[thm]{Conjecture}
\newtheorem{quest}[thm]{Question}

\theoremstyle{definition}
\newtheorem{defn}[thm]{Definition}

\theoremstyle{remark}
\newtheorem{rem}[thm]{Remark}
\makeatletter
\let\c@equation\c@thm
\makeatother
\numberwithin{equation}{section}

\DeclareMathOperator{\Sp}{Sp}
\DeclareMathOperator{\Hom}{Hom}
\DeclareMathOperator{\End}{End}

\DeclareMathOperator{\colim}{colim}
\DeclareMathOperator{\cA}{\mathcal{A}}
\DeclareMathOperator{\cC}{\mathcal{C}}
\DeclareMathOperator{\cD}{\mathcal{D}}

\DeclareMathOperator{\Id}{\mathrm{Id}}
\DeclareMathOperator{\Ext}{Ext}
\DeclareMathOperator{\Tor}{Tor}
\DeclareMathOperator{\Spec}{Spec}
\DeclareMathOperator{\Mod}{Mod}
\DeclareMathOperator{\Stable}{Stable}
\DeclareMathOperator{\Comod}{Comod}

\DeclareMathOperator{\fib}{fib}
\DeclareMathOperator{\cofib}{cofib}
\DeclareMathOperator{\Loc}{Loc}

\DeclareMathOperator{\Thick}{Thick}
\newcommand{\Thickid}{\Thick^\otimes}

\DeclareMathOperator{\coker}{coker}
\DeclareMathOperator{\Tot}{Tot}
\DeclareMathOperator{\Alg}{Alg}
\DeclareMathOperator{\Ind}{Ind}
\DeclareMathOperator{\Her}{Her}
\DeclareMathOperator{\Coh}{Coh}
\DeclareMathOperator{\pdim}{pdim}
\DeclareMathOperator{\Spc}{Spc}

\DeclareMathOperator{\Cotor}{Cotor}
\DeclareMathOperator{\Tel}{Tel}

\newcommand{\cal}{\mathcal}
\newcommand{\xr}{\xrightarrow}

\newcommand{\Z}{\mathbb{Z}}

\Crefname{figure}{Figure}{Figures}
\Crefname{assu}{Assumption}{Assumptions}
\Crefname{lem}{Lemma}{Lemmas}

\newcommand{\cL}{\mathcal{L}}
\newcommand{\cS}{\mathcal{S}}
\newcommand{\cT}{\mathcal{T}}
\newcommand{\cG}{\mathcal{G}}

\newcommand{\F}{\mathbb{F}}

\newcommand{\bc}[1]{\langle#1\rangle}

\title{Algebraic chromatic homotopy theory for $BP_*BP$-comodules}

\author{Tobias Barthel}
\address{Department of Mathematical Sciences, University of Copenhagen, Universitetsparken 5, 2100 K{\o}benhavn {\O}, Denmark}
\email{tbarthel@math.ku.dk}
\author{Drew Heard}
\address{Universit{\"a}t Hamburg, Bundesstrasse 55, 20146 Hamburg, Germany}
\email{drew.heard@uni-hamburg.de}

\date{\today}

\begin{document}

\begin{abstract}
In this paper, we study the global structure of an algebraic avatar of the derived category of ind-coherent sheaves on the moduli stack of formal groups. In analogy with the stable homotopy category, we prove a version of the nilpotence theorem as well as the chromatic convergence theorem, and construct a generalized chromatic spectral sequence. Furthermore, we discuss analogs of the telescope conjecture and chromatic splitting conjecture in this setting, using the local duality techniques established earlier in joint work with Valenzuela.
\end{abstract}
\maketitle

\tableofcontents
\def\biblio{}

\section{Introduction}

The chromatic approach to stable homotopy theory is a powerful tool both for  understanding the local and global structure of the stable homotopy as well as for making explicit computations. The goal of this paper is to study an algebraic version of this theory, based on the category of $BP_*BP$-comodules. As such, it is deeply intertwined with recent efforts to implement the chromatic perspective in motivic homotopy theory as well as in more algebraic contexts. 

More specifically, we work with a suitable version $\Stable_{BP_*BP}$ of the derived category of $BP_*BP$-comodules, which is an algebraic avatar of the category of ind-coherent sheaves on the moduli stack of formal groups. This category was introduced by Hovey \cite{hovey_htptheory,hovey_chromatic} and in related work of Palmieri~\cite{palmieri_memoir}, and then further studied by the authors and Valenzuela \cite[Sec.~8]{bhv1}. From an axiomatic point of view, $\Stable_{BP_*BP}$ is a prominent example of a non-Noetherian stable homotopy theory in the sense of~\cite{hps_axiomatic}, so that many of the standard techniques do not apply. The importance of this category is due to the fact that it sits at the intersection of three different areas, so that its local and global structure provides new insights in each of them:
\begin{enumerate}
	\item \emph{Stable homotopy theory: As an approximation to the category of spectra.} Many structural patterns of the stable homotopy category are visible through the lens of the Adams--Novikov spectral sequence. The $E_2$-term of this spectral sequence for the sphere spectrum is isomorphic to $\pi_* BP_*$ in $\Stable_{BP_*BP}$, so that this category is a very close algebraic approximation to the category of spectra. 
	In particular, the chromatic filtration in stable homotopy theory provides a filtration of $\Stable_{BP_*BP}$ by full subcategories $\Stable_{E(n)_*E(n)}$, whose suitably defined limit over $p$ is essentially equivalent to the limit of the category of $E(n)$-local spectra \cite{bss}.
	\item \emph{Algebraic geometry: The relationship to ind-coherent sheaves on the moduli stack of formal groups $\cal{M}_{fg}$.} By work of Quillen \cite{quillen_fg} there is a close connection between stable homotopy theory and the theory of formal groups. More specifically, our results can be translated into properties of the category of ind-coherent sheaves over a certain moduli stack $\cal{M}_{fg}$ of formal groups. The stack $\cal{M}_{fg}$ is stratified by height, and this height filtration corresponds to the chromatic filtration in stable homotopy. Thus, studying the category $\Stable_{E(n)_*E(n)}$ corresponds to geometrically studying ind-coherent sheaves on open substacks of $\cal{M}_{fg}$. One may therefore consider $\Stable_{BP_*BP}$ as a toy example of a category of ind-coherent sheaves on stratified stacks, which are for instance relevant in the geometric Langlands program~\cite{gaitsgory_indcoh}. 
		\item \sloppy \emph{Motivic homotopy theory: Motivic module spectra over the cofiber of $\tau$.} Via work of Isaksen \cite{isaksen2014stable}, $\Ext_{BP_*BP}^{\ast}(BP_*,BP_*)$ also appears naturally in motivic homotopy theory as the homotopy groups of $C\tau$, the motivic cofiber of $\tau$ over $\Spec(\mathbb{C})$.\footnote{Recall that, working in the $p$-complete setting, the motivic cohomology of a point over $\Spec(\mathbb{C})$ is isomorphic to $\F_p[\tau]$, where $\tau$ has bidegree $(0,1)$, and that this gives rise to an essential map $\tau \colon S^{0,-1} \to S^{0,0}$.} Joint work of Gheorghe--Xu--Wang \cite{gheorghe2017bp} shows that this isomorphism extends to an equivalence between $\Stable_{BP_*BP}$ and a category closely related to the category $\Mod_{C\tau}^{\mathrm{cell}}$ of cellular motivic $C\tau$-modules. Thus, our results about $\Stable_{BP_*BP}$ can be translated to results in the stable motivic homotopy category. 
\end{enumerate}

This exhibits $\Stable_{BP_*BP}$ as an important test case for the more in-depth study of related categories in these areas. 
	
\subsection*{Main results}
In \cite{mrw_77} Miller, Ravenel, and Wilson introduced the chromatic spectral sequence, which converges to the $E_2$-term of the Adams--Novikov spectral sequence. Based on systematic algebraic patterns seen in this work, Ravenel was lead to his famous nilpotence and periodicity conjectures \cite{ravenel_localization}, later proved by Devinatz, Hopkins, and Smith \cite{nilpotence1,nilpotence2}, giving rise to the field of chromatic homotopy theory. In this paper we develop and prove algebraic analogs of Ravenel's conjectures in the category of $BP_*BP$-comodules.

As noted previously we work with the category $\Stable_{BP_*BP}$ instead of the usual derived category $\cD_{BP_*BP}$. As is already clear from work of Hovey \cite{hovey_htptheory}, the usual derived category is homotopically poorly behaved; for example, the tensor unit $BP_*$ is not compact, and this necessitates working with the more complicated category $\Stable_{BP_*BP}$.

In order to construct $\Stable_{BP_*BP}$ we must first study the abelian category of $BP_*BP$-comodules. We do this in more generality in \Cref{sec:structurebp}, by recalling some basic properties of the abelian category of comodules over a flat Hopf algebroid. We quickly specialize to the case of $BP_*BP$ and $E(n)_*E(n)$, giving a classification of hereditary torsion theories for $\Comod_{E(n)_*E(n)}$. In \Cref{sec:stable} we recall the construction of the stable category $\Stable_{\Psi}$ associated to a flat Hopf algebroid, and give a change of rings theorem for Hopf algebroids associated to faithfully flat extensions. 

With the stable category associated to a flat Hopf algebroid defined, we move on to the study of the global structure of $\Stable_{BP_*BP}$ and $\Stable_{E(n)_*E(n)}$. On the abelian level, the structure of the category of $E(n)_*E(n)$-comodules is known to be much simpler when the prime is large compared to $n$. For example, $\pi_*E(n)_* \cong \Ext^s_{E(n)_*E(n)}(E(n)_*,E(n)_*)$ vanishes for $s> n^2+n$, whenever $p > n+1$. This is reflected in \Cref{thm:genericprimes}, where we prove the following; here we denote by $(K(n)_*,\Sigma(n))$ the Hopf algebroid studied extensively by Miller, Ravenel, and Wilson. 
\begin{thm*}
	 If $p > n+1$, then there is an equivalence $\Stable_{E(n)_*E(n)} \simeq \cD_{E(n)_*E(n)}$, between the stable category of $E(n)_*E(n)$-comodules and the usual derived category of $E(n)_*E(n)$-comodules. Similarly, if $n$ does not divide $p-1$, then there is an equivalence $\Stable_{\Sigma(n)} \simeq \cD_{\Sigma(n)}$. 
\end{thm*}

 In stable homotopy theory, the Morava $K$-theories $K(n)$ detect nilpotence.  In our algebraic setting we use the $BP_*BP$-comodule $\Tel(n)_* = v_n^{-1}BP_*/I_n$ as our detecting family for nilpotence, proving the following version of the nilpotence theorem in \Cref{sec:nilpotence}. This result crucially relies on the use of the category $\Stable_{BP_*BP}$ instead of the derived category, as here $BP_*$ is compact. 

\begin{thm*}(Algebraic nilpotence theorem - weak version)\label{thm:intronil}
	\begin{enumerate}
		\item  Suppose $F,X \in \Stable_{BP_*BP}$ with $F$ compact, then a map $f\colon F \to X$ is smash nilpotent, i.e., $f^{(m)}=0$ for some $m \gg 0$, if $\Tel(n)_*\otimes_{BP_*}f = 0$ for all $0 \le n \le \infty$. 
		\item A self map $f \colon \Sigma^i F \to F$ for $F$ compact in $\Stable_{BP_*BP}$ is nilpotent, in the sense that $f^j \colon \Sigma^{ij} F \to F$ is null for some $j \gg 0$, if and only if $\Tel(n)_*\otimes_{BP_*}f$ is nilpotent for all $0 \le n \le \infty$. 
		\item Suppose $X \in \Stable_{BP_*BP}$, then a map $f\colon BP_* \to X$ is smash nilpotent if $\pi_*(\Tel(n)_*\otimes_{BP_*}f) = 0$ for all $0 \le n \le \infty$. 
		\item Let $R$ be a ring object in $\Stable_{BP_*BP}$. Then an element $\alpha \in \pi_*R \cong \Ext_{BP_*BP}(BP_*,R)$ is nilpotent if and only if $\pi_*(\Tel(n)_* \otimes_{BP_*}\alpha)$ is nilpotent for all $0 \le n \le \infty$. 
\end{enumerate}
\end{thm*}
We call this a weak version of the algebraic nilpotence theorem, because the results are not as strong as those in \cite{nilpotence2}. Indeed, they do not account for all periodic elements in $\pi_*BP_*$, but only those appearing in Adams--Novikov filtration 0, which is a reflection of the fact that $\Tel(n)_*$ is not a field object. Indeed, the nilpotence theorem implies that there is a vanishing curve on the $E_{\infty}$-page of  the Adams--Novikov spectral sequence for the sphere that has slope tending to zero as $t-s$ approaches $\infty$. However, this vanishing curve is not present on the $E_2$-page and in fact there are non-nilpotent elements of positive Adams--Novikov filtration. It follows that there are many more non-nilpotent elements in $\Stable_{BP_*BP}$ than in stable homotopy theory. This complicates the structure of $\Stable_{BP_*BP}$; there appear to be many more thick subcategories than in stable homotopy theory. We will return to the systematic study of self maps and thick subcategories of $\Stable_{BP_*BP}$ in forthcoming work with Achim Krause. 

 One formulation of the telescope conjecture in stable homotopy is that $\bc{\Tel(m)} = \bc{K(m)}$ \cite{hovey_csc}, where $\bc{\Tel(m)}$ denotes the Bousfield class of the telescope of a finite spectrum of type $m$. Since $K(n)_*$ is not a $BP_*BP$-comodule, strictly speaking this question does not make sense in $\Stable_{BP_*BP}$. Nonetheless, it is a $BP_*$-module, and so one can formulate a variant of the telescope conjecture, which we show in  \Cref{thm:moravak} does hold. This gives some explanation for the use of $\Tel(n)_*$ in the nilpotence theorem above. 

\begin{thm*}
For all $n\ge 0$, there is an identity of Bousfield classes $\bc{K(n)_*} = \bc{\Tel(n)_*}$.
\end{thm*}

 In \Cref{sec:localduality} we move on to the study of the local structure of $\Stable_{BP_*BP}$. We begin by constructing localization functors $L_n$ for $0 \le n \le \infty$ which in particular define an exhaustive filtration of the full subcategory of compact objects. Their essential images $L_n\Stable_{BP*BP}$ are algebraic counterparts of the categories of $E(n)$-local spectra, which in turn form the building blocks of chromatic homotopy theory. In geometric terms, $L_n$ corresponds to the restriction to an open substack of $\cal{M}_{fg}$. Such functors have previously been studied by Hovey and Strickland \cite{hovey_chromatic,hs_localcohom,hs_leht}, who proved that there is an equivalence of categories between $L_n\Stable_{BP_*BP}$ and $\Stable_{E(n)_*E(n)}$.  As Hovey points out in \cite[p.~171]{hovey_chromatic} an alternative formulation of the telescope conjecture, namely that $L_n$ is the same as the Bousfield localization with the homology theory associated with $E(n)_*$, is false in general in $\Stable_{BP_*BP}$, however we note that this holds when $n < p-1$, see \Cref{rem:algtelescoperev}. 

The algebraic localization functors $L_n$ assemble into the algebraic chromatic tower
\[
\xymatrix{\ldots \ar[r] & L_2 \ar[r] & L_1 \ar[r] & L_0,}
\]
precisely as in stable homotopy theory. Hopkins and Ravenel have shown \cite{orangebook} that a compact spectrum is the limit of its chromatic tower. We prove the following variant of this in \Cref{sec:chromaticconvergence}. 

\begin{thm*}[Chromatic convergence]\label{thm:introcc}
If $M\in \Stable_{BP_*BP}$ has finite projective dimension, then there is a natural equivalence $M \simeq \lim_n L_{n}M$.
\end{thm*}
In particular, this implies that all compact objects of $\Stable_{BP_*BP}$ satisfy chromatic convergence. The strength of this algebraic chromatic convergence theorem is akin to that of the first author's generalization of the chromatic convergence theorem in stable homotopy theory \cite{barthel_cc}. We also show that $\lim_n L_n \simeq L_{\infty}$ where the latter is the localization functor associated to $BP_*/I_{\infty} \cong \Z/p$.

The Bousfield--Kan spectral sequence associated to the chromatic tower in stable homotopy leads to a spectral sequence of the form $E_1 = \pi_kM_nS^0 \implies \pi_kS^0$, where $M_n$ is the fiber of $L_n \to L_{n-1}$. Associated to the algebraic chromatic tower, we can similarly construct a spectral sequence. This recovers, and indeed generalizes, the classical chromatic spectral sequence, which is obtained by setting $X = Y = S^0$. 
\begin{thm*}[The chromatic spectral sequence]
 	For any spectra $X,Y$, there is a natural convergent spectral sequence
\[
E_1^{n,s,t} = \Ext^{s,t}_{BP_*BP}(BP_*X,M_{{n}}BP_*Y) \implies \Ext^{s,t}_{BP_*BP}(BP_*X, L_{\infty}BP_*Y). 
\]
Furthermore, if $BP_*Y$ satisfies the conditions of \Cref{thm:introcc}, then the spectral sequence converges to $\Ext_{BP_*BP}(BP_*X,BP_*Y)$.
 \end{thm*}
 
  By truncating the chromatic tower, we can also build a height $n$ analog of the chromatic spectral sequence which, as a special case, recovers the truncated chromatic spectral sequence constructed by Hovey and Sadofsky \cite[Thm.~5.1]{hov_sadofsky}.

As a concrete application of our results, we obtain the following transchromatic comparison between the $E_2$-terms of the $BP$-Adams spectral sequence and the $E$-Adams spectral sequence at height $n$, see \Cref{cor:comparision}. 
\begin{cor*}
If $X$ is a $p$-local bounded below spectrum such that $BP_*X$ has projective $BP_*$-dimension $\pdim(BP_*X) \le r$, then the natural map
\[
\xymatrix{\Ext_{BP_*BP}^s(BP_*,BP_*X) \ar[r] & \Ext_{E_*E}^s(E_*,E_*X)}
\]
is an isomorphism if $s < n-r-1$ and injective for $s = n-r-1$. 
\end{cor*}
A related result can be found in work of Goerss \cite[Thm.~8.24]{goerss_quasi-coherent_2008}, however the authors are unaware of a result of this generality in the literature. 

\subsection*{Relation to other work}
The present paper is a natural continuation of work of Hovey and Strickland \cite{hs_leht,hs_localcohom,hovey_chromatic} as well as unpublished work of Goerss \cite{goerss_quasi-coherent_2008}. In contrast to our algebraic approach, Goerss works more geometrically, studying the derived category of quasi-coherent sheaves on the moduli stack of $p$-typical formal group laws. However, both approaches are equivalent and consequently some of our results are equivalent to those obtained by Goerss.  For example, Goerss's chromatic convergence theorem \cite[Thm.~8.22]{goerss_quasi-coherent_2008} translates into a special case of \Cref{thm:chromaticconvergence2}. Similar geometric approaches have been studied by Hollander \cite{hollander}, Naumann \cite{naumann_stack_2007}, Pribble \cite{pribble}, Sitte \cite{sitte2014local}, and Smithling \cite{smithling}.

\subsection*{Conventions}\label{sec:conventions}
In this paper we work with stable $\infty$-categories in the quasi-categorical setting developed by Joyal~\cite{joyalqcat} and Lurie~\cite{htt,ha}. For simplicity, we will refer to a quasi-category as an $\infty$-category throughout this paper.

Unless otherwise noted, all categorical constructions are implicitly considered derived. For example, the tensor product $\otimes$ usually refers to the derived tensor product, and limits and colimits mean homotopy limits and homotopy colimits, respectively. The symbol $\boxtimes$ is reserved for the underived tensor product.

If $\cC$ is a closed symmetric monoidal stable $\infty$-category, the internal function object will be denoted by $\iHom_{\cC}$ to distinguish it from the merely spectrally enriched categorical mapping object $\Hom_{\cC}$. This is related to the usual mapping space via a natural weak equivalence $\Omega^\infty \Hom_{\cC}(X,Y) \simeq \operatorname{Map}_{\cC}(X,Y)$. If no confusion is likely to arise, the subscript $\cC$ will be omitted from the notation. 

When dealing with chain complexes, we will always  employ homological grading, i.e., complexes are written as 
\[
\xymatrix{\ldots \ar[r]^-d & X_{1} \ar[r]^-d & X_0 \ar[r]^-d & X_{-1} \ar[r]^-d & \ldots}
\]
with the differential $d$ lowering degree by 1. As usual, taking cohomology of a chain complex $X$ reverses the sign of the homology, that is $H^*(X) = H_{-*}(X)$.

We work with Hopf algebroids $(A,\Psi)$ over a commutative ring $K$ throughout; that is, $A$ and $\Psi$ are both commutative $K$-algebras. We will always assume that $\Psi$ is a flat $A$-module, and we call such Hopf algebroids flat. 

\subsection*{Acknowledgments}
This work was inspired in part by the aforementioned unpublished manuscript of Goerss \cite{goerss_quasi-coherent_2008}. We would like to thank Paul Goerss, Sharon Hollander, Mark Hovey, Achim Krause, and Gabriel Valenzuela for helpful conversations on the subject of this paper. Moreover, we are grateful to Gabriel Valenzuela for useful comments on an earlier draft of this document. This project began at the Max Planck Institute for Mathematics in Bonn, which we thank for its hospitality. The first-named author was partially supported by the DNRF92. The second-named author was partially supported by the SPP1786.

\section{Hopf algebroids and the structure of $\Comod_{BP_*BP}$}\label{sec:structurebp}
In this section we will prove some basic results about the abelian category $\Comod_{\Psi}$ of comodules over a flat Hopf algebroid $(A,\Psi)$. We assume the reader is familiar with the notion of comodules over a Hopf algebroid, for which good references include \cite[App.~A]{greenbook} and \cite{hovey_htptheory}.  We finish with a classification of the hereditary torsion theories for Landweber exact $BP_*$-algebras of height $n$, extending work of Hovey and Strickland \cite{hs_leht}.

We note that since we work with abelian categories in this section, all functors are assumed to be underived. 
\subsection{Recollections on Hopf algebroids and comodules}
Given a flat Hopf algebroid $(A,\Psi)$ over a commutative ring $K$ (i.e., $A$ and $\Psi$ are commutative $K$-algebras), we will write $\Comod_{\Psi}$ for the abelian category of $\Psi$-comodules. The following proposition, which is essentially a compendium of results in \cite[Sec.~1]{hovey_htptheory}, establishes the basic properties of the category $\Comod_{\Psi}$. 

\begin{prop}\label{prop:comodbasics}
The abelian category $\Comod_{\Psi}$ of comodules over a flat Hopf algebroid $(A,\Psi)$ is a complete and cocomplete locally presentable Grothendieck abelian category with a closed symmetric monoidal structure. A $\Psi$-comodule is compact or dualizable if and only if the underlying $A$-module is finitely presented or finitely presented and projective, respectively. Moreover, the forgetful functor
\[
\xymatrix{\epsilon_*\colon \Comod_{\Psi} \ar[r] & \Mod_A}
\]
is exact, faithful, symmetric monoidal, and preserves all colimits. The corresponding right adjoint $\epsilon^*$, which sends an $A$-module $M$ to the cofree $\Psi$-comodule $\Psi \otimes M$, is exact as well.
\end{prop}
Given $\Psi$-comodules $M$ and $N$, we write $M \otimes_A N$ for the monoidal product, and $\iHom_{\Psi}(M,N)$ for the internal Hom object. We will often omit the subscript if it is clear from context. No confusion should arise with the use of $M \otimes_A N$; for example, given a $\Psi$-comodule $M$, $\Psi \otimes_A M$ could be interpreted as the extended comodule on $M$ or as the symmetric monoidal product of the comodules $\Psi$ and $M$, but these turn out to be the naturally isomorphic, see \cite[Lem.~1.1.5]{hovey_htptheory}. 

We say that a $\Psi$-comodule $I$ is relatively injective if $\Hom_{\Psi}(-,I)$ takes $A$-split short exact sequences to short exact sequences. 

\begin{lem}\label{lem:injcomodules}
A $\Psi$-comodule $M$ is injective if and only if it is a retract of an extended comodule $\Psi \otimes I$ on an injective $A$-module $I$. A comodule $M$ is relatively injective if and only if it is a retract of an extended comodule. 
\end{lem}
\begin{proof}
	This is well known, see for example \cite[Lem.~2.1]{hs_localcohom} for the first statement, and \cite[Lem.~3.1.2]{hovey_htptheory} for the latter. 
\end{proof}

The link between Hopf algebroids and topology arises from the observation that if $F$ is a ring spectrum with $F_*F$ flat over $F_*$, then $(F_,F_*F)$ is a flat Hopf algebroid, and $F_*X$ is an $F_*F$-comodule for any spectrum $X$. We will be particularly interested in Hopf algebroids that are Landweber exact over $BP_*$ in the following sense, see \cite[Def.~2.1 and Def.~4.1]{hs_leht}

\begin{defn}\label{defn:bpcomodule}
	Suppose $f \colon BP_* \to E_*$ is a ring homomorphism, then $E_*$ is said to be a Landweber exact $BP_*$-algebra of height $n \in \mathbb{N} \cup \{\infty\}$ if the following conditions are satisfied:
	\begin{enumerate}
	 	\item The $BP_*$-module $E_*/I_n$ is non-zero, and $E_*/I_k \cong 0$ for all $k > n$. If $E_*/I_n$ is non-zero for all $n$, then the height is set to be $\infty$. 
	 	\item The functor from $BP_*BP$-comodules to $E_*$-modules induced by $M \mapsto E_* \otimes_{BP_*} M$  is exact. 
	 \end{enumerate} 
\end{defn}

Typical examples include Johnson--Wilson theories $E(n)_*$, Morava $E$-theory $(E_n)_*$, and $v_n^{-1}BP_*$, all of which have height $n$, see also \Cref{sec:torsen}. Given such an $E_*$, we can associate a Hopf algebroid $(E_*,E_*E) = (BP_*,E_* \otimes_{BP_*}BP_*BP \otimes_{BP_*}E_*)$. By \cite[Thm.~C]{hs_leht} any two Landweber exact $BP_*$-algebras of the same height have equivalent categories of comodules. 

Recall that a flat Hopf algebroid $(A,\Psi)$ is an Adams Hopf algebroid if $\Psi$ is a filtered colimit of comodules $\Psi_i$ that are finitely generated and projective as $A$-modules. We then have the following, which is proved in \cite[Sec.~1.4]{hovey_htptheory}.  

\begin{prop}[Hovey]\label{prop:adamsha}
The Hopf algebroids $(BP_*,BP_*BP)$ and $(E_*,E_*E)$, where $E_*$ is any height $n$ Landweber exact $BP_*$-algebra, are Adams Hopf algebroids, i.e., the corresponding category of comodules is generated by the dualizable comodules. 
\end{prop}

\subsection{The cotensor product and Cotor}
In this section we recall some basic facts about the cotensor product and its derived functor $\Cotor$. Our approach is slightly non-standard, in that we prefer to work with a relative version of $\Cotor$, which agrees with the one defined using standard homological algebra only when the first variable is flat. 

To begin, recall that given a right $\Psi$-comodule $M$ and a left $\Psi$-comodule $N$, the cotensor product $M \square_{\Psi} N$ is defined as the equalizer
\[
\xymatrix{M \square_\Psi N \ar[r] &  M \otimes_A N \ar@<0.5ex>[r]^-{\psi_M \otimes 1}  \ar@<-0.5ex>[r]_-{ 1 \otimes \psi_N} & M \otimes_A \Psi \otimes_A N.}
\]
Note that this only inherits the structure of a $K$-module. 

\begin{lem}\label{lem:cotorextended}
	If $N = \Psi \otimes_A N'$ is an extended comodule, then $M \square_{\Psi}N \cong M \otimes_A N'$.
\end{lem}
\begin{proof}
	It is easy to check that the map $M \otimes_A N' \xr{\Psi_M \otimes 1} M \otimes_A \Psi \otimes_ A N'$ is the kernel of $\psi_M \otimes 1 - 1 \otimes \psi_N$. 
\end{proof}
The following definition naturally arises when using the methods of relative homological algebra, see \cite[p.~15]{em_relative} or \cite[Sec.~2]{barthel2014e_2}. 

\begin{defn}
 A proper injective resolution of a comodule $M$ is a resolution of $M$ by relative injectives such that each map in the resolution is split as map of $A$-modules. 
\end{defn}

Such resolutions always exist; indeed, for a comodule $M$, the standard cobar resolution
\[
C^\ast_{\Psi}(M) = \left ( \xymatrix{M \ar[r]^-{\psi_M} & \Psi \otimes_A M \ar@<0.5ex>[r]^-{\Delta \otimes 1} \ar@<-0.5ex>[r]_-{1 \otimes \Psi_M} & \Psi \otimes_A \Psi \otimes_A M \ar@<0.75ex>[r] \ar@<-0.75ex>[r] \ar[r] & \cdots } 
\right )
\]
is a resolution by relative injectives, which is split using the map $\epsilon \colon \Psi \to A$. Moreover, such resolutions are unique up to chain homotopy \cite[Thm.~2.2]{gm_dgha}. 

\begin{defn}\label{defn:cotorunderived}
Given $\Psi$-comodules $M$ and $N$, let $0 \to M \xr{i} J^\bullet$ and $0 \to N \xr{i'} L^\bullet$ be proper injective resolutions. We define
\[
\Cotor_\Psi^n(M,N) = H^n(\Tot^{\oplus}(J^\bullet \square_{\Psi} L^\bullet)),
\]
where $\Tot^{\oplus}$ is the totalization of the bicomplex with respect to the direct sum.
\end{defn}

The next result says that it is enough to take a resolution of either of the variables. 
\begin{lem}\label{lem:cotorbalanced}
	The maps
	\[
\xymatrix{
J^\bullet \square_{\Psi} N \ar[r]^{1 \square i'} & J^\bullet \square_{\Psi} L^\bullet & M \square_{\Psi} J^\bullet \ar[l]_{i \square 1}}
	\]
	induce isomorphisms on homology. 
\end{lem}
\begin{proof}
	This is proved in the context of comodules over a coalgebra in \cite{em_coalgebras}.  By \Cref{lem:injcomodules} each $J^i$ is a retract of an extended comodule $\tilde{J}^i \otimes_A \Psi$, so we have $J^i \square_{\Psi} N \cong \tilde{J}^i \otimes_A N$ by \Cref{lem:cotorextended}. It follows that $J^i \square_{\Psi} (-)$ preserves $A$-split exact sequences. 

	Filter $J^\bullet$ by $F_n(J^\bullet) = J^{\le n}$, which induces filtrations on $J^\bullet \square_{\Psi} N$ and $J^\bullet \square_{\Psi} L^\bullet$. One checks, using the fact that $J^i \square_{\Psi} (-)$ preserves $A$-split exact sequences, that in the associated spectral sequence the map $J^\bullet \square_{\Psi} N \xr{1 \square i} J^\bullet \square_{\Psi} L^\bullet$ induces an isomorphism on $E_1$-terms, and so is an isomorphism on homology. The argument for $i \square 1$ is similar. 
\end{proof}
\begin{rem}
	Note that this result implies that $\Cotor^*_\Psi(M,N)$ is independent of the choice of resolution of $M$ or $N$. 
\end{rem}
Since $\Comod_\Psi$ has enough injectives, it is more customary to define $\Cotor_\Psi^*(M,N)$ by taking an injective resolution $I^*$ of $N$ to construct the derived functors of $M \square_{\Psi} -$. Let us temporarily write $\widetilde{\Cotor}_{\Psi}(M,N)$ for this functor. 
\begin{lem}
	For $M,N \in \Comod_\Psi$ with $M$ flat, then there is an isomorphism
	\[
\Comod^*_\Psi(M,N) \cong \widetilde {\Comod}^*_{\Psi}(M,N).
	\]
\end{lem}
\begin{proof}
	This is proved in \cite[Lem.~A1.2.8]{greenbook}: Ravenel assumes $M$ projective, but it is clear from the proof that $M$ flat is sufficient.  
\end{proof}
We prefer to use the relative version of $\Cotor$ since it allows us to dispense with flatness hypothesis in certain results, such as \Cref{lem:cotorcomparison}.

Given a left (respectively right) $\Psi$-comodule, we can always turn it into a right (respectively left) $\Psi$-comodule, by conjugating the action by the antipode $\chi$ of $\Psi$. We use that implicitly in the next result. 
\begin{lem}\label{lem:cotortwist}
Suppose given two comodules $M,N \in \Comod_{\Psi}$, then there is a natural isomorphism
\[
M \Box_{\Psi} N \xr{\sim} A \Box_{\Psi} (M \otimes_A N)
\]
of $K$-modules. 
\end{lem}
\begin{proof}
	This follows by comparing the coequalizers defining the two cotensor products, and a careful diagram chase. We note that if we write $\psi_M(m) = \Sigma_i m_i \otimes x_i$ and $\psi_N(n) = \Sigma_j y_j \otimes n_j$, then the comodule structure map on $M \otimes_A N$ is given by $\psi_{M \otimes_A N}(m \otimes n) = \Sigma_{i,j} (\chi(x_i)y_j \otimes m_i \otimes n_j)$. 
\end{proof}
This leads to the following. 

\begin{lem}\label{lem:cotorcomparison}
Suppose given two comodules $M,N \in \Comod_{\Psi}$, then there is a natural isomorphism
\[
\xymatrix{\Cotor_{\Psi}^*(M,N) \ar[r]^-{\sim} & \Cotor_{\Psi}^*(A,M\otimes_AN)}
\] 
of $K$-modules. 
\end{lem}

\begin{proof}
As noted above, given a comodule $X$, the $\Psi$-cobar complex $C_{\Psi}^*(X)$ is a proper injective resolution, and so can be used to compute $\Cotor$. The isomorphism $M \otimes_A N \cong M \Box_{\Psi} (\Psi \otimes_A N)$ of \Cref{lem:cotorextended}, along with \Cref{lem:cotortwist} shows that
\[
\begin{split}
M \Box_{\Psi} C^{k}_{\Psi}(N) &= M \Box_{\Psi} (\Psi \otimes_{A} \Psi^{\otimes k} \otimes_A N) \\
							& \cong M \otimes_A \Psi^{\otimes k} \otimes_A N \\
							& \cong \Psi \Box_{\Psi}(M \otimes_A \Psi^{\otimes k} \otimes_A N) \\
							& \cong A \Box_{\Psi} (M \otimes_A \Psi^{\otimes (k+1)} \otimes_A N).
\end{split}
\]
for all $k$. This leads to a quasi-isomorphism
\[
M \Box_{\Psi} C_{\Psi}^*(N) \simeq A \Box_{\Psi} C_{\Psi}^*(M \otimes_A N),
\] 
hence the desired isomorphism of Cotor groups. 
\end{proof}

\begin{rem}
If we were to use $\widetilde \Cotor_\Psi^*(M,N)$ instead of $\Cotor_\Psi^*(M,N)$, then we only know how to prove this when $M$ is a flat $A$-module. 
\end{rem}

\subsection{Hereditary torsion theories}
In this subsection we give a brief introduction to hereditary torsion theories, and prove a result relating hereditary torsion theories under certain localizations of categories. We use the terminology of hereditary torsion theories in order to distinguish it from the notion of localizing subcategory used in the context of stable $\infty$-categories in later sections. 
\begin{defn}\label{defn:hereditarytorsion}
	Let $\cA$ be a cocomplete abelian category. A full subcategory $\cT$ of $\cA$ is said to be a hereditary torsion theory if it is closed under subobjects, quotient objects, extensions, and arbitrary coproducts in $\cA$. 
\end{defn}
We recall that given a class of maps $\cal{E}$ in a category $\cA$, we say that an object $M \in \cA$ is $\cal{E}$-local if $\Hom_{\cA}(f,M)$ is an isomorphism for all $f \in \cal{E}$, and we denote the full subcategory of $\cal{E}$-local objects by $L_{\cal{E}}\cA$. Given such a class of maps it is known (for example by \cite{MR521257}) that there exists a localization functor $L \colon \cA \to \cA$ such that for each $M \in \cA$ we have $LM \in L_{\cal{E}}\cA$. 

Given a hereditary torsion theory $\cT$, we let $\cal{E}_\cT$ denote the class of $\cT$-equivalences, i.e., those maps whose kernel and cokernel are in $\cT$. 
\begin{defn}\label{defn:gabriellocal}
	Let $\cA$ be an abelian category and $\cT$ a hereditary torsion theory in $\cA$, then the Gabriel localization $L_{\cT} \colon \cA \to \cA$ is the localization functor associated to the class $\cal{E}_{\cT}$ of ${\cT}$-equivalences.
\end{defn}
\begin{thm}
Suppose $\cA$ is a Grothendieck abelian category. There is a natural bijection between hereditary torsion theories of $\cA$ and Gabriel localizations of $\cA$: 
\begin{align*}
\Her(\cA) & \longleftrightarrow    \Loc^G(\cA)  \\
\cT &  \longmapsto  L_{\cT} \\
\ker(L) & \longmapsfrom L.
\end{align*}
This bijection is realized by sending a hereditary torsion theory $\cT$ to the localization $L_{\cT}$ defined above; conversely, a Gabriel localization functor $L$ determines a hereditary torsion theory $\cT_L = \ker(L)$. 
\end{thm}
\begin{proof}
	See \cite[Thm.~1.13.5]{Borceux_2}. 
\end{proof}
\begin{prop}\label{prop:quotienthtt}
Suppose $\cT \subseteq \cA$ is a hereditary torsion theory, and let $\cA/{\cT}$ be the associated local category with localization functor $\Phi_*\colon \cA \to \cA/{\cT}$. If $\cS \subseteq \cA/{\cT}$ is a hereditary torsion theory in $\cA/\cT$, then there exists a hereditary torsion theory $\overline{\cS} \subseteq \cA$ with $\cT \subseteq \overline{\cS}$ and such that $\cS = \Phi_*(\overline{\cS})$. 
\end{prop}
\begin{proof}
Write for $(\Phi_*,\Phi^*)$ and $(F_{\cS},G_{\cS})$ for the localization adjunction corresponding to $\cT$ and $\cS$, respectively, so that we have a diagram
\[
\xymatrix{\cA \ar@<0.5ex>[r]^-{\Phi_*} & \cA/\cT \ar@<0.5ex>[r]^-{F_{\cS}} \ar@<0.5ex>[l]^-{\Phi^*} & (\cA/\cT)/\cS. \ar@<0.5ex>[l]^-{G_{\cS}}}
\]
We first claim that $(\Phi_*^{\cS},\Phi^*_{\cS}) = (F_{\cS}\Phi_*,\Phi^*G_{\cS})$ is a localization adjunction. Indeed, $\Phi_*^{\cS}$ is exact and there are natural isomorphisms
\[
\xymatrix{\Phi_*^{\cS}\Phi^*_{\cS} = F_{\cS}\Phi_*\Phi^*G_{\cS} \ar[r]^-{\sim} & F_{\cS}G_{\cS} \ar[r]^-{\sim} & \Id.}
\]
Therefore, there exists a hereditary torsion theory $\overline{\cS} \subseteq \cA$ corresponding to $(\Phi_*^{\cS},\Phi^*_{\cS})$; in particular, $\cT = \ker(\Phi_*) \subseteq \ker(\Phi_*^{\cS}) = \overline{\cS}$.

It thus remains to show that $\cS = \Phi_*(\overline{\cS})$. Clearly, $F_{\cS}\Phi_*(\overline{\cS}) = \Phi_*^{\cS}(\overline{\cS}) = 0$, so $\Phi_*(\overline{\cS}) \subseteq \ker(F_{\cS}) = \cS$. Conversely, for $X \in \cS$ we have
\[
0 = F_{\cS}X = F_{\cS}\Phi_*\Phi^*X = \Phi_*^{\cS}\Phi^*X,
\]
which implies $\Phi^*X \in \ker(\Phi_*^{\cS}) = \overline{\cS}$. Consequently, $X = \Phi_*\Phi^*X \in \Phi_*(\overline{\cS})$, hence $\cS \subseteq \Phi_*(\overline{\cS})$.
\end{proof}

\begin{rem}
With notation as above, $\cA/\cT$ inherits the structure of a Grothendieck abelian category, see \cite[Cor.~4.6.2]{popescu}. 
\end{rem}
\subsection{Height $n$ cohomology theories and classification of hereditary torsion theories}\label{sec:torsen}
In this section we introduce the Hopf algebroids $(E(n)_*,E(n)_*E(n))$ and $(K(n)_*,\Sigma(n))$ closely related to $(BP_*,BP_*BP)$, and give a classification of the hereditary torsion theories of the former. 

In stable homotopy theory the geometric counterpart of the Brown--Peterson spectrum is the moduli stack of $p$-typical formal group laws. If we restrict to open substacks of formal group laws of height at most $n$, then the corresponding spectrum is Johnson--Wilson $E$-theory $E(n)$, with coefficient ring
\[
E(n)_* \cong \Z_{(p)}[v_1,\ldots,v_{n-1},v_n^{\pm 1}], 
\]
with $|v_i| =2(p^i-1)$. This gives rise to a flat Hopf algebroid $(E(n)_*,E(n)_*E(n))$ where, by Landweber exactness of $E(n)$, we have $E(n)_*E(n) \cong E(n)_* \otimes BP_*BP \otimes E(n)_*$.

The geometric point associated to the open substack corresponds to Morava $K$-theory $K(n)$, whose coefficient ring is the graded field
\[
K(n)_* \cong \F_{p^n}[v_n^{\pm 1}].
\]
 Importantly $K(n)_*$ satisfies a K{\"u}nneth isomorphism: for spectra $X$ and $Y$ there is an isomorphism
 \begin{equation}\label{eq:knkunneth}
K(n)_*(X \otimes Y) \cong K(n)_*X \otimes_{K(n)_*}K(n)_*Y. 
 \end{equation}

Correspondingly one would expect to study the Hopf algebroid $(K(n)_*,K(n)_*K(n))$. However, $K(n)_*$ is not Landweber exact, and it turns out to be slightly more convenient to work with the Hopf algebroid $(K(n)_*,\Sigma(n))$ with $\Sigma(n) = K(n)_* \otimes_{BP_*} BP_*BP \otimes_{BP_*} K(n)_*$; note that if $K(n)_*$ were Landweber exact, then this would be precisely $(K(n)_*,K(n)_*K(n))$. For example, it is the latter Hopf algebroid that appears in the important change of rings theorem of Miller and Ravenel \cite[Thm.~2.10]{miller_ravenel_local}. 

As noted previously, $E(n)_*$ is an example of a Landweber exact $BP_*$-algebra of height $n$ in the sense of \Cref{defn:bpcomodule}. Other examples include $v_n^{-1}{BP_*}$ or $E_*$, where $E=E_n$ denotes the $n$-th Morava $E$-theory, with coefficient ring
\[
E_* \cong \mathbb{W}(\F_{p^n})[\![u_1,\ldots,u_{n-1}]\!][ u^{\pm 1}],
\]
where $|u_i| = 0$ and $|u| = -2$. Geometrically this corresponds to the universal deformation of the geometric point associated to Morava $K$-theory. Since the associated comodule categories are equivalent, the hereditary torsion theories are equivalent for any Landweber exact $BP_*$-algebra of height $n$.  

The geometry of the moduli stack of formal groups is reflected in the global structure of the associated Hopf algebroids, more precisely in the poset of their hereditary torsion theories. A (partial) classification of hereditary torsion theories for $\Comod_{BP_*BP}$ was proved by Hovey and Strickland \cite{hs_leht}, following the classification of thick subcategories (or Serre classes) of finitely presented $BP_*BP$-comodules by Jeanneret, Landweber, and Ravenel \cite{jlr_thick}.  We use this classification of hereditary torsion theories for $BP_*BP$-comodules and the results of the previous subsection to classify the hereditary torsion theories for $E_*E$-comodules.

In what follows, let $\cT_n$ denote the full subcategory of all graded $BP_*BP$-comodules that are $v_n$-torsion, with the convention that $\cT_{-1} = \Comod_{BP_*BP}$. The next result is \cite[Thms.~B and C]{hs_leht}.

\begin{thm}[Hovey--Strickland]\label{thm:bphtt}
Let $\cT \subseteq \Comod_{BP_*BP}$ be a hereditary torsion theory containing a nontrivial compact comodule, then $\cT = \cT_n$ for some $-1 \le n$. Moreover, if $n\ge 0$, then the local category corresponding to $\cT_n$ is naturally equivalent to $\Comod_{E_*E}$ with $E_*$ a Landweber exact $BP_*$-algebra of height $n$. 
\end{thm}

Fix $-1 \le n$ and let 
\[
\xymatrix{\Comod_{BP_*BP} \ar@<0.5ex>[r]^-{\Phi_*} & \Comod_{E_*E} \ar@<0.5ex>[l]^-{\Phi^*}}
\]
be the localization adjunction corresponding to $\cT_n$.

\begin{cor}\label{cor:ehtt}
Let $\cS \subseteq \Comod_{E_*E}$ be a hereditary torsion theory, then $\cS = \cT_m$ for some $-1 \le m \le n$. 
\end{cor}
\begin{proof}
By \Cref{thm:bphtt} and \Cref{prop:quotienthtt}, there exists a hereditary torsion theory $\overline{\cS} \subseteq \Comod_{BP_*BP}$ such that $\cT_n^{BP} \subseteq \overline{\cS}$ and $\cS = \Phi_*(\overline{\cS})$. The first property implies that $\overline{\cS}$ contains a nonzero compact $BP_*BP$-comodule, hence \Cref{thm:bphtt} shows that there exists an $m$ with $\overline{\cS} = \cT_m^{BP}$. By definition, $\Phi_*\cT_m^{BP} = \cT_m^{E}$, so the claim follows. 
\end{proof}
We note that, in contrast to \Cref{thm:bphtt}, we do not require that $\cS$ contains a nontrivial compact comodule; instead, this condition is automatically satisfied in this case. 

\section{The Stable category of comodules}\label{sec:stable}
In this section we study the stable category of $\Psi$-comodules, previously introduced in \cite{hovey_htptheory} and \cite{bhv1}. In particular, we use a derived version of the cotensor product considered in the last section to derive Ravenel's base-change spectral sequence for $\Cotor$ \cite[App.~A.1.3.11]{greenbook}, as well as a variant of a change of rings theorem of Hovey and Sadofsky \cite[Thm.~3.3]{hov_sadofsky}. 

\subsection{The definition of $\Stable_{\Psi}$}
As noted by Hovey \cite{hovey_htptheory}, the category of chain complexes of comodules should be thought of as like topological spaces, in the sense that there is both a notion of homology and homotopy, and to form the `correct' version of the derived category we should invert the homotopy, not homology, isomorphisms. In \emph{loc.~cit.}~Hovey constructed such a category $\Stable_\Psi$ associated to a Hopf algebroid\footnote{Actually, Hovey constructed $\Stable_{\Psi}$ for amenable Hopf algebroids, see \cite[Def.~2.3.2]{hovey_htptheory}, but all the Hopf algebroids we consider in this paper are amenable.} $(A,\Psi)$. In \cite[Sec.~4]{bhv1} we gave an alternative construction, which agrees with Hovey's model under some very mild conditions on the Hopf algebroid. We give a brief review of our construction here, referring the reader to \cite{bhv1} for the details.

For some motivation, we start with an observation of Hovey \cite[Sec.~3]{hovey_chromatic}: in the derived category of $BP_*BP$-comodules the tensor unit $BP_*$ is not compact (essentially due to the existence of non-nilpotent elements in $\Ext_{BP_*}(BP_*,BP_*)$). The idea of the following definition is to force the tensor unit (and indeed, all dualizable comodules) to be compact. Thus, let $(A,\Psi)$ be a flat Hopf algebroid, and write $\cG = \cG_{\Psi}$ for the set of dualizable $\Psi$-comodules and $\cD_{\Psi}$ for the usual derived category of comodules. 
\begin{defn}\label{defn:stable_comod}
We define the stable $\infty$-category of $\Psi$-comodules as the ind-category of the thick subcategory of $\cD_{\Psi}$ generated by $\cG$ viewed as complexes concentrated in degree $0$, i.e.,
\[
\Stable_{\Psi} = \Ind(\Thick_{\Psi}(\cG)).
\]
\end{defn}
The next proposition summarizes some basic properties of $\Stable_{\Psi}$. Proofs are given in \cite[Sec.~4]{bhv1}. 
\begin{prop}\label{prop:stableprops}
	Let $(A,\Psi)$ be a flat amenable Hopf algebroid.
	\begin{enumerate}
		\item $\Stable_\Psi$ is a presentable stable $\infty$-category compactly generated by $\cG$, equipped with a closed symmetric monoidal product preserving colimits in both variables. 
		\item There is a cocontinuous functor $\omega \colon \Stable_\Psi \to \cD_\Psi$ to the ordinary derived category, which is a (symmetric monoidal) equivalence when $(A,A)$ is a discrete Hopf algebroid. 
		\item The functor $\omega$ is given by Bousfield localization at the homology isomorphisms, i.e., those morphisms which induce an isomorphism on $H_*$. 
		\item For a Hopf algebroid $(A,\Psi)$ there is a canonical equivalence between $\Stable_{\Psi}$ and the underlying $\infty$-category of the model category constructed by Hovey in \cite{hovey_htptheory}.
		\item There is an adjunction of stable categories
\[
\xymatrix{\epsilon_*\colon\Stable_{\Psi} \ar@<0.5ex>[r]^-{} & \cD_A\colon \epsilon^* \ar@<0.5ex>[l]}
\]
extending the adjunction from \Cref{prop:comodbasics}. 
	\end{enumerate}
\end{prop}
Since Point (3) is perhaps not clearly outlined in \cite{bhv1}, we note that it follows from Hovey's construction of $\Stable_{\Psi}$ and Point (4) above. 

As noted in \Cref{prop:stableprops}, $\Stable_{\Psi}$ is compactly generated by the set of (isomorphism classes of) dualizable $\Psi$-comodules. We will say that it is monogenic if it is compactly generated by $A$ itself. The following implies that many of the categories we study in this paper are monogenic; in particular $\Stable_{BP_*BP}$ itself is. 
\begin{prop}\cite[Cor.~6.7]{hovey_htptheory}\label{prop:monogenic}
	If $E$ is a ring spectrum that is Landweber exact over $MU$ or $BP$ and $E_*E$ is commutative, then $\Stable_{E_*E}$ is monogenic. 
\end{prop}

Given $M,N \in \Stable_{\Psi}$ we will again write $M \otimes_{A} N$ for the monoidal product, and $\iHom_{\Psi}(M,N)$ for the internal Hom object (sometimes we will omit the subscripts if the context is clear). Because of \Cref{prop:stableprops}, we always assume our Hopf algebroids are amenable. 

For technical reasons, it is sometimes useful to restrict to a certain subclass of Hopf algebroids. 
\begin{defn}\cite[Def.~4.14]{bhv1}
	Let $(A,\Psi)$ be a flat Hopf algebroid, and write $\Comod^\omega_\Psi[0]$ for the image of the nerve of the abelian category of compact comodules in $\cD_\Psi$. We call $(A,\Psi)$ a Landweber Hopf algebroid if $\Comod_\Psi^\omega[0]$ is contained in $\Thick_\Psi(A)$.
\end{defn}
This definition includes all the commonly used Hopf algebroids in stable homotopy theory, see \cite[Sec.~4.3]{bhv1}. The next result was mentioned without proof in \cite[Rem.~4.30]{bhv1}.
\begin{lem}\label{lem:lamonogenic}
	If $(A,\Psi)$ is a Landweber Hopf algebroid, then $\Stable_{\Psi}$ is monogenic.
\end{lem}
\begin{proof}
	It suffices to show that $\Thick_{\Psi}(A) = \Thick_{\Psi}(\cG)$; we always have $\Thick_{\Psi}(A) \subseteq \Thick_{\Psi}(\cG)$, and so we must show the other inclusion. Let $\cD_0 \subset \cD_{\Psi}$ be the full subcategory of complexes $Q$ with homology concentrated in finitely many degrees such that $H_d(Q) \in \Comod_{\Psi}$ is compact. In \cite[Lem.~4.16]{bhv1} we showed that $\cD_0 = \Thick_{\Psi}(A)$. But since $G \in \cD_0$ for each $G \in \cG$ there is an inclusion $\Thick_{\Psi}(\cG) \to \cD_0 = \Thick_{\Psi}(A)$, completing the lemma. 
\end{proof}
Landweber Hopf algebroids have another important property, which rests on a theorem due to Krause~\cite{krause_deriving}.
 
\begin{prop}\label{prop:landweberha}
Assume $(A, \Psi)$ is a Landweber Hopf algebroid with $A$ coherent. There is
a natural $t$-structure on $Stable_\Psi$ such that the inclusion functor $\iota \colon  \cD_{\Psi} \to \Stable_{\Psi}$ is $t$-exact and induces natural equivalences
\[
\xymatrix{ \cD_{\Psi}^{\le k} \ar[r]^-{\sim} &\Stable_{\Psi}^{\le k}}
\]
on the full subcategories of $k$-coconnective objects for all $k \in \Z$. The inverse equivalence is given by inverting the homology isomorphisms. 
\end{prop}
\begin{proof}
	This is the content of \cite[Prop.~4.17]{bhv1}, where we proved this under the hypothesis that $A$ is Noetherian. This can be generalized to the case that $A$ is coherent using the work of Krause \cite{krause_deriving}, as extended to the $\infty$-categorical setting by Lurie \cite[App.~C.5.8]{sag}
\end{proof}

\begin{defn}
Let $\Stable_{\Psi}^{< \infty}$ be the full subcategory of those $M \in \Stable_{\Psi}$ for which there exists some $k$ such that $M \in \Stable_{\Psi}^{\le k}$.
\end{defn}

Since $\Stable_{\Psi}$ is a stable $\infty$-category, $\Hom_{\Psi}(A,M)$ canonically has the structure of a spectrum. To avoid confusion in the following definition, we write $\pi^{\mathrm{st}}_*$ for the homotopy groups of a spectrum. 

\begin{defn}
For $M \in \Stable_\Psi$, we define the homotopy groups of $M$ as $\pi_*M = \pi^{\mathrm{st}}_*\Hom_\Psi(A,M)$. 	
\end{defn}

\begin{rem}\label{rem:cotor}
By \cite[Prop.~6.10]{hovey_htptheory}, $\pi_*A \cong \Ext_\Psi^*(A,A)$, so that $\pi_*M$ is always a graded module over the graded-commutative ring $\Ext_\Psi^*(A,A)$. More generally given any discrete $\Psi$-comodule $M$, thought of as an object of $\Stable_\Psi$, Hovey's result shows that $\pi_*M \cong \Ext_\Psi^*(A,M)$. By \cite[A1.1.6]{greenbook} this is in turn isomorphic to $\Cotor_\Psi^*(A,M)$. 
\end{rem}

The relation between homology and homotopy in $\Stable_{\Psi}$ is given by the following:

\begin{lem}\label{lem:homrep}
	For any $M \in \Stable_{\Psi}$ we have $\pi_*(\Psi \otimes M) \cong H_*M$. 
\end{lem}
\begin{proof}
	This follows easily by the adjunction between $\Stable_{\Psi}$ and $\cD_A$ stated in \Cref{prop:stableprops}(5); indeed, we have 
\[
\Hom_{{\Psi}}(N,\Psi \otimes M) \simeq \Hom_{\cD_A}(\epsilon_* N,M)
\]
for $N \in \Stable_{\Psi}$ and $M \in \cD_A$, so that in particular $\pi_*(\Psi \otimes M) \cong H_*M$.    
\end{proof}
\subsection{Some derived functors}
Given a morphism $\Phi \colon (A,\Psi) \to (B,\Sigma)$ of Hopf algebroids, there exists a functor $\Phi_* \colon \Comod_{\Psi} \to \Comod_{\Sigma}$ induced by $M \mapsto B \otimes_AM$, with a right adjoint $\Phi^*$. We shall see in the next lemma that both of these exist in the associated stable categories and that, interestingly, there is a third adjoint. 
\begin{lem}
If $\Phi \colon (A,\Psi) \to (B,\Sigma)$ is a map of Hopf algebroids, then there exist adjoint functors
\[
\xymatrix{\Stable_{\Psi} \ar@<-1ex>[r]_-{\Phi_*} \ar@<1ex>[r]^-{\Phi_!} & \Stable_{\Sigma} \ar[l]|-{\Phi^*}}
\]
where $\Phi_*$ is left adjoint to $\Phi^*$, which in turn is left adjoint to $\Phi_!$.
\end{lem}
\begin{proof}
	 If $M$ is finitely generated and projective over $A$ (and hence dualizable in $\Comod_\Psi$, see \Cref{prop:comodbasics}), then $B \otimes_A M$ is finitely-generated and projective over $B$, so that 
	 \[
	 \xymatrix{B\otimes_A -\colon \Comod_{\Psi} \ar[r] & \Comod_{\Sigma}}
	 \] 
	 preserves dualizable comodules. It follows that there is an induced exact functor 
	 $\Phi_*\colon \Thick_{\Psi}{(\cG_{\Psi})} \to \Thick_{\Psi}{(\cG_{\Sigma})}$. Applying $\Ind$, we get a functor $\Phi_* \colon \Stable_{\Psi} \to \Stable_{\Sigma}$ that preserves all colimits and compact objects. Thus, by \cite[Prop.~5.3.5.13]{htt} and \cite[Thm.~1.7]{bds}, $\Phi_*$ has a right adjoint $\Phi^*$, which has a further right adjoint $\Phi_!$. 
\end{proof}

The canonical map from the initial Hopf algebroid $(K,K)$ to any Hopf algebroid $(A,\Psi)$ will always be denoted by $\gamma_{\Psi}\colon (K,K) \to (A,\Psi)$; if the Hopf algebroid is clear from context, the subscript $\Psi$ will be omitted. We now give a simple proof of the fact that $\gamma_{\Psi}^*$ is the functor of derived primitives.

\begin{lem}\label{lem:gammaformula}
	For $M \in \Stable_{\Psi}$ there is a natural equivalence 
	\[
\gamma_{\Psi}^*M \simeq \Hom_{{\Psi}}(A,M). 
	\]
\end{lem}
\begin{proof}
	By \Cref{prop:stableprops}(2) there is a symmetric monoidal equivalence of $\infty$-categories $\Stable_K \simeq \cD_K$. Let $M \in \Stable_{\Psi}$, then
	\[
\gamma_{\Psi}^\ast M \simeq \Hom_{\cD_K}(K,\gamma_{\Psi}^\ast M) \simeq \Hom_{{\Psi}}((\gamma_{\Psi})_*K,M) \simeq \Hom_{{\Psi}}(A,M).\qedhere
	\]
\end{proof}
Note that by definition we have $\pi_*M = \pi_*^{\mathrm{st}}(\gamma_{\Psi}^\ast M)$. Moreover, given a map $\Phi \colon (A,\Psi) \to (B,\Sigma)$ there is a commutative diagram of Hopf algebroids
\[
\xymatrix{
	(K,K) \ar[r]^{\gamma_{\Psi}}  \ar[dr]_{\gamma_{\Sigma}}& (A,\Psi) \ar[d]^{\Phi} \\
	 & (B,\Sigma),
}
\]
so that $\gamma_{\Sigma}^* \simeq \gamma_{\Psi}^\ast \Phi^*$. 

The next result is known as the projection formula.

\begin{lem}[Projection formula]\label{lem:projformula}
For $M \in \Stable_{\Sigma}$ and $N \in \Stable_{\Psi}$, there is a natural equivalence
\[
\xymatrix{(\Phi^*M)\otimes_A N \ar[r]^-{\sim} & \Phi^*(M \otimes_B \Phi_*(N)).}
\]
\end{lem}
\begin{proof}
The map is constructed as the adjoint of the natural transformation
\[
\xymatrix{\Phi_*(\Phi^*(M) \otimes_A N) & \Phi_*\Phi^*(M) \otimes_B \Phi_*N \ar[l]_-{\sim} \ar[r]^-{\epsilon \otimes \Id} & M \otimes_B \Phi_*N,}
\]
where $\epsilon$ is the counit of the adjunction $(\Phi_*,\Phi^*)$. Since all functors involved preserve colimits, it suffices to verify the claim for $M=B$ and $N=A$, for which it is clear. 
\end{proof}
We can give an explicit formula for the right adjoint $\Phi^*$. 
\begin{lem}\label{lem:pushforwardformula}
For $\Phi\colon (A,\Psi) \to (B,\Sigma)$ a map of Hopf algebroids, the right adjoint $\Phi^*$ of $\Phi_*$ can be identified as the derived primitives of the extended $\Psi$-comodule functor, i.e., 
\[
\Phi^*M \simeq \Hom_{\Sigma}(B,M \otimes_A \Psi),
\]
for any $M \in \Stable_{\Sigma}$. 
\end{lem}
\begin{proof}
We first note that the statement makes sense: $M \otimes_A \Psi$ obtains the structure of a $\Sigma$-comodule via the comodule structure on $M$. It is also clearly a $\Psi$-comodule, and $\Hom_{\Sigma}(B,M \otimes_A \Psi)$ obtains a $\Sigma$-comodule structure by an argument similar to \cite[1.3.11(a)]{greenbook}.

There are natural equivalences 
\begin{align*}
\gamma_{\Sigma}^*(M\otimes_A \Psi)  & \simeq \gamma_{\Sigma}^*(M\otimes_B \Phi_*(\Psi)) & (\text{since } \Phi_*(\Psi) \simeq B \otimes_A \Psi) \\
& \simeq \gamma_{\Psi}^*\Phi^*(M \otimes_{B} \Phi_*(\Psi)) & (\text{since } \gamma_{\Sigma}^* \simeq  \gamma_{\Psi}^*\Phi^*) \\
& \simeq \gamma_{\Psi}^*(\Phi^*(M)\otimes_A \Psi) & (\text{by \Cref{lem:projformula}}) \\
& \simeq \gamma_{\Psi}^*\epsilon^*(\Phi^*(M)) & (\text{since } \epsilon^*(-) = \Psi \otimes_A -)\\
& \simeq \Phi^*(M) 
\end{align*}
where $\epsilon^*$ is as in \Cref{prop:stableprops}(5). The same argument as in \cite[1.3.11(a)]{greenbook} shows that these equivalences are compatible with the comodule structures.
\end{proof}

In virtue of \Cref{lem:cotorcomparison}, the following definition is a natural generalization of the classical construction of the Cotor groups of discrete comodules. 

\begin{defn}\label{defn:cotor}
We define the derived cotensor product of any two objects $M,N \in \Stable_{\Psi}$ as the derived primitives of their tensor product, 
\[
\Cotor_{\Psi}(M,N) = \gamma_{\Psi}^*(M \otimes_A N),
\]
viewed as an object of $\Stable_{K}$. 
\end{defn}
We then define $\Cotor^i_\Psi(M,N) = \pi_i^{\mathrm{st}}\Cotor_{\Psi}(M,N) = \pi_i(M \otimes_A N)$. If $M$ and $N$ are discrete comodules then, by \Cref{rem:cotor,lem:cotorcomparison}, this agrees with the definition of $\Cotor$ given in \Cref{defn:cotorunderived}. Furthermore:
\begin{lem}\label{lem:iHomhomotopy}
	If $(A,\Psi)$ is a Landweber Hopf algebroid with $A$ coherent, then for $M,N \in \Stable_\Psi^{< \infty}$ we have $\pi_*\iHom_{\Psi}(M,N) \cong \Ext_{\Psi}^\ast(\omega_*M,\omega_* N)$. 
\end{lem}
\begin{proof}
	\sloppy By adjunction $\pi_*\iHom_{\Psi}(M,N) \cong \pi^{\mathrm{st}}_* \Hom_{\Psi}(M,N)$. Now apply \cite[Cor.~4.19]{bhv1}, using \Cref{prop:landweberha}. 
\end{proof}

As an easy application of the results of this section, we can reinterpret the base-change spectral sequence for Cotor constructed by Ravenel in \cite[App.~A.1.3.11]{greenbook}. Note that we can dispense of the hypothesis that $M$ is flat by our use of the relative Cotor functor. 

\begin{cor}
Let $f\colon (A,\Psi) \to (B,\Sigma)$ be a map of Hopf algebroids. If $M$ is a discrete (right) $\Psi$-comodule and $N$ is a discrete (left) $\Sigma$-comodule, then there is a natural convergent spectral sequence
\[
\Cotor_{\Psi}^s(M,\Cotor_{\Sigma}^t(B \otimes_A \Psi,N)) \implies \Cotor_{\Sigma}^{s+t}(M \otimes_A B,N)
\]
with differentials $d^r\colon E_r^{s,t} \to E_r^{s+r,t-r+1}$. 
\end{cor}
\begin{proof}
First, using \Cref{lem:gammaformula} and \Cref{lem:pushforwardformula}, we obtain equivalences
\[
\gamma_{\Sigma}^*((B \otimes_A \Psi) \otimes_B N) \simeq \Hom_{\Sigma}(B,(B \otimes_A \Psi) \otimes_B N) \simeq \Hom_{\Sigma}(B,N \otimes_A \Psi) \simeq  f^*N.
\]
The projection formula \Cref{lem:projformula} then gives natural equivalences
\begin{align*}
\gamma_{\Sigma}^*(f_*(M) \otimes_B N) & \simeq \gamma_{\Psi}^*f^*(f_*(M) \otimes_B N) \\
& \simeq \gamma_{\Psi}^*(M \otimes_A f^*(N)) \\
& \simeq \gamma_{\Psi}^*(M \otimes_A (\gamma_{\Sigma}^*((B\otimes_A\Psi) \otimes_B N))).
\end{align*}
By testing on extended $\Sigma$-comodules as in \cite[Sec.~6]{barthel2014e_2}, the Grothendieck spectral sequence associated to the two functors
\[
\xymatrix{ \gamma_{\Psi}^*(M \otimes_A -) & \text{and} & \gamma_{\Sigma}^*((B\otimes_A\Psi) \otimes_B -)}
\]
exists and converges \cite[Thm.~5.8.3]{weibel_homological}. By construction and \Cref{lem:cotorcomparison}, the resulting spectral sequence recovers the Cotor spectral sequence. 
\end{proof}

\begin{rem}[Geometric interpretation]
We recall from \cite{naumann_stack_2007} that to a flat Hopf algebroid $(A,\Psi)$ we can associate an algebraic stack $\mathfrak{X}$ with a fixed presentation $\Spec(A) \to \mathfrak{X}$, and that this gives rise to an equivalence of 2-categories between flat Hopf algebroids and rigidified algebraic stacks \cite[Thm.~8]{naumann_stack_2007}. Moreover, there is an equivalence of abelian categories between $\mathrm{QCoh}(\mathfrak{X})$, the category of quasi-coherent sheaves on $\mathfrak{X}$, and $\Comod_{\Psi}$. Using this we can define the category $\Ind\Coh_{\mathfrak{X}}$ of ind-coherent sheaves on $\mathfrak{X}$, and show that it is equivalent to $\Stable_{\Psi}$, see \cite[Prop.~5.40]{bhv1}. This equivalence is symmetric monoidal. Geometrically, this means that Cotor as defined in \Cref{defn:cotor} corresponds to the derived global sections of the tensor product of ind-coherent sheaves. 
\end{rem}
\subsection{Change of rings}
As another application, we will prove a change of rings theorem for Hopf algebroids associated to faithfully flat extensions. For precursors of this result, see \cite[Prop.~3.2]{baker_elliptic}, \cite[Thm.~3.3]{hov_sadofsky}, \cite[Thm.~6.2]{hs_leht} and \cite[Thm.~D]{hovey_morita}.

Given a Hopf algebroid $(A,\Psi)$ and morphism $\Phi \colon A \to B$ of $K$-algebras, let $\Sigma_\Phi = B \boxtimes_A \Psi \boxtimes_A B$, where we use the symbol $\boxtimes$ to denote the underived tensor product. Note that $(B,\Sigma_\Phi)$ forms a Hopf algebroid, and there is a natural morphism of Hopf algebroids $\Phi \colon (A,\Psi) \to (B,\Sigma_{\Phi})$. In general $(B,\Sigma_{\Phi})$ need not be a flat Hopf algebroid, even when $(A,\Psi)$ is. It is, however, when $B \boxtimes_A \Psi$ is a flat $A$-module. 

\begin{lem}\label{lem:ffconservative}
Suppose $(A,\Psi)$ is a Landweber Hopf algebroid with $A$ coherent. If $T$ is a faithfully-flat $A$-module, then the composite
\[
\xymatrix{\Stable_{\Psi}^{<\infty} \ar[r]^-{\epsilon_*} & \Stable_{A}^{<\infty} \ar[r]^-{T \otimes_A-} & \cD_A}
\]
is conservative.
\end{lem}
\begin{proof}
By \Cref{prop:landweberha} there is an equivalence of categories $\Stable_\Psi^{<\infty} \simeq \cD_\Psi^{<\infty}$, so $\epsilon_*$ restricted to $\Stable_\Psi^{<\infty}$ is conservative. Now let $f\colon M \to N$ be a morphism in $\Stable_{A}^{<\infty}\simeq\cD_A^{<\infty}$ such that $T \otimes_A f$ is an equivalence. The morphism $f$ gives rise to a cofiber sequence $M \xr{f} N \to \cofib(f)$ where, by assumption, $T \otimes_A \cofib(f) \simeq 0$. Since $T$ is faithfully flat over $A$, this implies that $\cofib(f) \simeq 0$, so that $f$ was an equivalence to begin with. 
\end{proof}

For the following compare \cite[Thm.~6.2]{hs_leht}. 
\begin{prop}
	Let $\Phi \colon A \to B$ be as above, and suppose that $(A,\Psi)$ is a Landweber Hopf algebroid with $A$ coherent. Suppose the composite
	\[
A \xr{\eta_R} \Psi \xr{1 \otimes \Phi}\Psi \otimes_A B
	\]
is a faithfully flat extension of $A$, then $\Phi^*$ induces an equivalence
\[
\xymatrix{\Stable_{\Sigma_{\Phi}} \ar[r]^-{\sim} & \Stable_{\Psi}.} 
\]
\end{prop}
\begin{proof}
For this proof, we use the notation $M \boxtimes N$ to denote the underived tensor product between two modules $M$ and $N$. 

We will first show that the unit $u \colon \text{id} \to \Phi^*\Phi_*$ is an equivalence. Since $\Phi_*$ and $\Phi^*$ preserve all colimits and $\Stable_{\Psi}$ is monogenic by \Cref{lem:lamonogenic}, the unit $u$ is a natural equivalence if and only if it is so when evaluated on $A$. Moreover, $u_A\colon A \to \Phi^*\Phi_*A$ is a map between objects in $\Stable_{\Psi}^{\le 0}$ and so by \Cref{lem:ffconservative} it suffices to show that $(\Psi \otimes_A B) \otimes_A u_A$ is an equivalence. To see this, first observe that the projection formula \Cref{lem:projformula} together with \Cref{lem:pushforwardformula} give
\[
\begin{split}
(\Psi \otimes_A B) \otimes_A \Phi^* \Phi_* A 
& \simeq \Phi^*(\Phi_*(\Psi \otimes_A B) \otimes_B \Phi_* A) \\
& \simeq \Phi^*\Phi_*(\Psi \otimes_A B) \\
& \simeq \Hom_{\Sigma_{\Phi}}(B, \Psi \otimes_A (B \otimes_A \Psi \otimes_A B)).
\end{split}
\]
Note that $B \otimes_A \Psi \simeq B \boxtimes_A \Psi$ since $\Psi$ is flat over $A$. Then, since $B \boxtimes_A \Psi$ is assumed to be flat over $A$, we deduce an equivalence $B \otimes_A \Psi \otimes_A B \simeq B \boxtimes_A \Psi \boxtimes_A B = \Sigma_{\Phi}$. Thus, we have
\[
\begin{split}
(\Psi \otimes_A B) \otimes_A \Phi^* \Phi_* A & \simeq \Hom_{\Sigma_{\Phi}}(B,\Psi \otimes_A \Sigma_{\Phi}) \\
& \simeq \Hom_{\Sigma_{\Phi}}(B,(\Psi \otimes_A B) \otimes_B \Sigma_{\Phi}) \\
& \simeq \Hom_{B}(B,\Psi \otimes_A B) \\
& \simeq \Psi \otimes_A B.
\end{split}
\]
It is standard to verify that this equivalence is induced by $u_A$, i.e., $(\Psi \otimes_A B) \otimes_A u_A$ is an equivalence, as required. 

Let $c$ denote the counit of the adjunction $(\Phi_*,\Phi^*)$ and suppose $Y \in \Stable_{\Sigma_{\Phi}}$. In order to show that $c\colon \Phi_*\Phi^*Y \to Y$ is an equivalence, it suffices to prove that the top morphism in the following commutative diagram is an equivalence
\[
\xymatrixcolsep{5pc}
\xymatrix{\Hom_{\Sigma_{\Phi}}(B, \Phi_*\Phi^*Y) \ar[r]^-{\Hom(B,c_{Y})} \ar[d]_{\sim} & \Hom_{\Sigma_{\Phi}}(B,Y) \ar[d]^{\sim} \\
\Hom_{\Psi}(A,\Phi^*\Phi_*\Phi^*Y) \ar[r]_-{\Hom(A,\Phi^*c_{Y})} & \Hom_{\Psi}(A,\Phi^*Y),}
\]
since $B = \Phi_*A$ is a compact generator of $\Stable_{\Sigma_{\Phi}}$. Because the unit of the adjunction is an equivalence, the triangle identity implies that the bottom horizontal map is an equivalence as well, and the claim follows. 
\end{proof}

\begin{rem}
This demonstrates how working systematically on the derived level can help to considerably simplify arguments, cf.~the proof of \cite[Thm.~D]{hovey_morita}. 
\end{rem}

\section{Morava theories and generic primes}\label{sec:morava}

In this section we study the stable categories $\Stable_{\Sigma(n)}$ and $\Stable_{E_*E}$ associated to the Hopf algebroids ($K(n)_*,\Sigma(n))$ and $(E_*,E_*E)$ introduced in \Cref{sec:torsen}, proving that for certain primes they are equivalent to their respective derived categories. In particular, we show that this is true whenever $p$ is large with respect to $n$. This implies that in these cases the stable category of comodules is much simpler, an algebraic manifestation of the well-known fact that chromatic homotopy at height $n$ simplifies when the prime $p$ is much larger than $n$. 

Recall that the homology theory $E_*$ is complex-oriented and the associated formal group law over $E_*$ is the universal deformation of the Honda formal group, the formal group law associated to Morava $K$-theory. We define the Morava stabilizer group $\mathbb{S}_n$ to be the group of automorphisms of the Honda formal group law of height $n$. If $n$ is not divisible by $p-1$, then $\mathbb{S}_n$ is of finite cohomological dimension $n^2$, which implies that $\Ext^{s,t}_{\Sigma(n)}(K(n)_*,K(n)_*)$  is zero for $s > n^2$. This leads to the following definition, where as usual $E_*$ denotes any height $n$ Landweber exact $BP_*$-algebra. 
\begin{defn}
	For any $n$, the set of $K(n)$-generic primes is the set of primes $p$ for which $n$ is not divisible by $p-1$, and the set of $E$-generic primes is the intersection of the sets of $K(i)$-generic primes for $0 \le i \le n$. 
\end{defn}

In the case of Morava $E$-theory, the $E$-based chromatic spectral sequence can be used to show that if $p$ is an $E$-generic prime, then $\Ext^{s,t}_{E_*E}(E_*,E_*)=0$ for $s > n^2+n$ \cite[Thm.~5.1]{hov_sadofsky}. The main result of this section is that the natural functors
\[
\Stable_{\Sigma(n)} \to \cD_{\Sigma(n)} \quad \text{ and  } \quad \Stable_{E_*E} \to \cD_{E_*E}
\]
are equivalences for the set of $K(n)$-generic and $E$-generic primes, respectively. Note that such a statement is not true for $\Stable_{BP_*BP}$ since, for example, $BP_*$ is compact in $\Stable_{BP_*BP}$ but not in $\cD_{BP_*BP}$. This is shown by Hovey \cite[Sec.~3]{hovey_chromatic} using the existence of non-nilpotent elements in $\Ext_{BP_*BP}^{s,t}(BP_*,BP_*)$ of positive cohomological degree. 
\subsection{Field theories}

Let $(K,\Upsilon)$ be a Hopf algebroid over a field $K$, so that $\Upsilon$ is in fact a Hopf algebra over $K$. There are two important types of examples. First, for any finite group $G$, the group ring of $G$ over the field $k$ has the structure of a Hopf algebra, so that $(k,kG)$ is a Hopf algebroid. Secondly, for any field object $K$ in the category of spectra, $(K_*,K_*K)$ is a Hopf algebroid over $K_*$. In particular, we can consider the Steenrod algebra $(\F_p,\mathcal{A}_*)$ and $(K(n)_*,\Sigma(n))$ corresponding to $H\F_p$ and Morava $K$-theory $K(n)$ for a given prime $p$ and height $n\ge 0$, respectively. As a consequence of the nilpotence theorem, these are essentially all fields of the stable homotopy category \cite[Prop.~1.9]{nilpotence2}.

The following two lemmata generalize Corollary 1.2.10 and Lemma 1.3.9 of  \cite{palmieri_memoir}. As is standard, we define the homology theory associated to $E \in \Stable_{\Upsilon}$ via the assignment $X \mapsto \pi_*(E \otimes X)$. 

\begin{lem}\label{lem:extendedcomodulessplit}
Let $(K,\Upsilon)$ be a Hopf algebroid over a field $K$. 
\begin{enumerate}
	\item The homology theory represented by $\Upsilon$ is ordinary (chain) homology $H_*$, and this satisfies the K\"unneth formula. 
	\item For any $M \in \Stable_{\Upsilon}$, $\Upsilon \otimes M$ decomposes as a direct sum of suspensions of $\Upsilon$. 
\end{enumerate}
\end{lem}
\begin{proof}
That $\Upsilon$ represents homology is a special case of \Cref{lem:homrep}. Since $\pi_*(\Upsilon \otimes M) \cong H_*M$ is a free graded $K$-module, it satisfies the K\"unneth formula, and so (1) holds. Moreover,	 we can construct a map 
\[
\xymatrix{\bigoplus_{b\in H_*M}\Sigma^{|b|}\Upsilon \ar[r] & \Upsilon \otimes M}
\]
in $\Stable_{\Upsilon}$, where the direct sum is indexed by a $K$-basis of $H_*M$. By construction, this map is an equivalence, and (2) follows. 
\end{proof}

\begin{lem}\label{lem:ftlocsubcatcrit}
Assume that $\Stable_{\Upsilon}$ is monogenic and suppose that $\cD$ is a localizing subcategory of $\Stable_{\Upsilon}$ containing a non-acyclic object $M_0$, then $\Loc(\Upsilon) \subseteq \cD$.  
\end{lem}
\begin{proof}
Because $\Stable_{\Upsilon}$ is monogenic, the localizing ideals coincide with the localizing subcategories. Since $0 \not\simeq \Upsilon \otimes M_0 \in \cD$, we get $\Upsilon \in \cD$ by \Cref{lem:extendedcomodulessplit}.
\end{proof}
In order to apply this to the examples of interest, we need the following. 
\begin{prop}\label{prop:sigmamonog}
	The category $\Stable_{\Sigma(n)}$ is monogenic. 
\end{prop}
\begin{proof}
In this proof we will again use the symbol $\boxtimes$ to denote the underived tensor product.

	Let $N \in \Stable_{\Sigma(n)}$ be compact. By construction of $\Stable_{\Sigma(n)}$, $N$ is in the thick subcategory generated by the dualizable $\Sigma(n)$-comodules. We will show that it is in the thick subcategory generated by $K(n)_*$. We note that by \Cref{prop:comodbasics} each dualizable discrete $\Sigma(n)$-comodule is finitely generated and projective as a $K(n)_*$-module, i.e., as a $K(n)_*$-module it is isomorphic to a finite direct sum of copies of $K(n)_*$, up to suspension. 

	Let $E$ be the Landweber exact cohomology theory with $E_* \cong \Z_p[v_1,\ldots,v_{n-1},v_n^{\pm 1}]$, so that $E_*/I_n \cong K(n)_*$ and $E_*E/I_n \cong K(n)_*E \cong \Sigma(n)$ \cite[p.~15]{hovey_morava_1999}. Let $f \colon (E_*,E_*E) \to (K(n)_*,\Sigma(n))$ denote the quotient morphism of Hopf algebroids. Then, for a $\Sigma(n)$-comodule $M$ there are equivalences
\begin{align*}
\Sigma(n) \boxtimes_{K(n)_*} M&\cong E_*E \boxtimes_{E_*} K(n)_* \boxtimes_{K(n)_*} M \\
& \cong E_*E \boxtimes_{E_*} M .
\end{align*}

In particular, $M$ is also an $E_*E$-comodule, with comodule structure map given by the composite
	\[
M \xr{\psi_M} \Sigma(n) \boxtimes_{K(n)_*} M \cong E_*E \boxtimes_{E_*} M.
	\]
 We will write $M^\sharp$ when we think of $M$ as an $E_*E$-comodule.

	For arbitrary $M \in \Stable^{< \infty}_{\Sigma(n)}$, \Cref{lem:pushforwardformula} gives equivalences
	\[
	\begin{split}
f^*M \simeq \Hom_{\Sigma(n)}(K(n)_*,M \otimes_{E_*}E_*E) &\simeq \Hom_{\Sigma(n)}(K(n)_*,M \boxtimes_{E_*}E_*E) \\
& \simeq \Hom_{\Sigma(n)}(K(n)_*,M \boxtimes_{K(n)_*} \Sigma(n)) \\
  & \simeq \Hom_{\Sigma(n)}(K(n)_*,M \otimes_{K(n)_*} \Sigma(n)) \\
& \simeq M,
\end{split}
	\]
with $E_*E$-comodule structure given as above, where we have used that our Hopf algebroids are flat. 

 It follows that $f^*M \simeq M^\sharp$, and in particular that $f^*(K(n)_*) \simeq K(n)^{\sharp}_* \simeq (E_*/I_n)^{\sharp}$. Since $I_n$ is a finitely-generated invariant ideal of $E_*$, it follows that $E_*/I_n$ is a finitely-presentable $E_*$-module, and hence so is $f^*(P) \simeq P^{\sharp}$ for any dualizable $\Sigma(n)$-comodule $P$. Then, $f^*(P)$ is dualizable in $\Stable_{E_*E}$, but because $\Stable_{E_*E}$ is monogenic by \Cref{prop:monogenic} and $E_*$ itself is compact, $f^*P$ is actually compact in $\Stable_{E_*E}$. 

 Since $f^*$ is exact, this implies that if $N \in \Stable_{\Sigma(n)}^{< \infty}$ is in the thick subcategory generated by the dualizable $\Sigma(n)$-comodules, then $f^*N \simeq N^{\sharp}$ is in the thick subcategory generated by the dualizable $E_*E$-comodules, i.e., $N^{\sharp}$ is compact in $\Stable_{E_*E}$. Again, using the fact that $\Stable_{E_*E}$ is monogenic, we see that $N^{\sharp}$ is in the thick subcategory generated by $E_*$. It follows that $f_*f^*N$ is in the thick subcategory generated by $K(n)_*$. Now we have cofiber sequences 
	\[
E_*/I_k \otimes K(n)_* \xr{ \cdot v_k} E_*/I_k \otimes K(n)_* \xr{\hspace{3mm}} E_*/I_{k+1} \otimes K(n)_*,
	\] 
	and since $K(n)_*$ is killed by $I_n$, these give rise to equivalences
	\[
E_*/I_k \otimes K(n)_* \simeq (E_*/I_{k-1} \otimes K(n)_*) \oplus (E_*/I_{k-1} \otimes \Sigma K(n)_*)
	\]
	for all $0 \le k \le n$, hence
	\[
	K(n)_* \otimes_{E_*} K(n)_* \cong E_*/I_n \otimes_{E_*} K(n)_* \cong \bigoplus_{0\le j \le n}\Sigma^j K(n)_*^{\lambda_n(j)},
	\]
	where $\lambda_n(j) = \binom{n}{j}$. Therefore,
	\[
	\begin{split}
f_*f^*N \simeq K(n)_* \otimes_{E_*} N  &\simeq K(n)_* \otimes_{E_*} K(n)_* \otimes_{K(n)_*} N\\ &
\simeq N \oplus \bigoplus_{1 \le j \le n} \Sigma^jN^{\oplus \lambda_n(j)}
\end{split}
	\]
	is in the thick subcategory generated by $K(n)_*$. It follows that $N \in \Thick(K(n)_*)$ as required.
\end{proof}

\subsection{Generic primes}
We now focus on the behavior of $\Stable_{\Sigma(n)}$ and $\Stable_{E_*E}$ at the set of $K(n)$-generic and $E$-generic primes, respectively. We start with $\Stable_{\Sigma(n)}$. 

\begin{lem}\label{prop:knresolution}
If $p-1 \nmid n$, then $K(n)_* \in \Thick_{\Sigma(n)}(\Sigma(n))$. 
\end{lem}
\begin{proof}
If $p-1 \nmid n$, the cohomological $p$-dimension of the Morava stabilizer group $\mathbb{S}_n$ is $n^2$, so there exists a length $n^2$ projective resolution
\[
\xymatrix{0 \ar[r] & P_{\bullet} \ar[r] & \Z_p \ar[r] & 0}
\]
of the trivial $\Z_p\llbracket \mathbb{S}_n\rrbracket$-module $\Z_p$, see \cite[Thm.~4]{henn_finite}. As shown in \emph{loc.~cit.}, this resolution can be lifted to a finite resolution of $K(n)_*$ as a $\Sigma(n)$-comodule, such that each term is a direct summand of a finite wedge of copies of $\Sigma(n)$. In the usual way, we can split the long exact sequence into short exact sequences. Starting from the final term and working our way back to $K(n)_*$, the claim follows inductively. 
\end{proof}

\begin{rem}\label{rem:kncdimsmallprimes}
For $p-1 \mid n$, while $\mathbb{S}_n$ has infinite cohomological $p$-dimension, it is still of virtual cohomological dimension $n^2$. In stable homotopy theory, this fact manifests itself in the existence of a finite spectrum $X_{p,n}$ of type $0$ such that $K(n)^*X_{p,n}$ has projective dimension $n^2$ over $\Sigma(n)^*$, see \cite[Proof of Thm.~8.9]{hovey_morava_1999}. Such complexes were constructed by Hopkins, Ravenel, and Smith as explained in \cite[Sec.~8.3]{orangebook}; note, however, that $X_{p,n}$ cannot be taken to be $S^0$ if $p-1 \mid n$.
\end{rem}

\begin{prop}\label{prop:kngeneric}
Suppose $p$ is a $K(n)$-generic prime, that is $p-1 \nmid n$, then the natural functor
\[
\xymatrix{\omega \colon\Stable_{\Sigma(n)} \ar[r]^-{\sim} & \cD_{\Sigma(n)}}
\]
is an equivalence of symmetric monoidal stable $\infty$-categories.
\end{prop}
\begin{proof}
The functor $\omega $ exhibits $\cD_{\Sigma(n)}$ as the localization of $\Stable_{\Sigma(n)}$ at the quasi-isomorphisms, i.e., it is localization at the localizing subcategory of all $M \in \Stable_{\Sigma(n)}$ such that $\pi_*(\Sigma(n) \otimes M) \cong H_*M =0$, see \Cref{prop:stableprops}. It follows from \Cref{prop:knresolution} that $\pi_*(M) \cong \pi_*(K(n)_* \otimes M) = 0$, hence $M\simeq 0$ since $\Stable_{\Sigma(n)}$ is monogenic by \Cref{prop:sigmamonog}. Therefore, $\omega$ is localization at $(0)$. 
\end{proof}

\begin{rem}
Combining the proof of \Cref{prop:kngeneric} with \Cref{rem:kncdimsmallprimes}, we see that, for any prime $p$, $M\in \Stable_{\Sigma(n)}$ being acyclic implies $\pi_*(K(n)_*X_{p,n} \otimes M) =0$. 
\end{rem}
For the following we let $\Spc(\Stable_{\Sigma(n)})$ denote the Balmer spectrum associated to $\Stable_{\Sigma(n)}$, see \cite{balmer_spec}. 
\begin{cor}
Suppose $p-1 \nmid n$. If $\cD \subseteq \Stable_{\Sigma(n)}$ is a localizing subcategory containing a non-zero object $M_0$, then $\cD = \Stable_{\Sigma(n)}$. In particular, there are no nontrivial proper thick subcategories in $\Stable_{\Sigma(n)}^{\omega}$, i.e., $\Spc(\Stable_{\Sigma(n)}) = \{\ast\}$.
\end{cor}
\begin{proof}
By \Cref{prop:kngeneric} such an $M_0$ corresponds to a non-acyclic object of $\Stable_{\Sigma(n)}$. Hence, combining \Cref{lem:ftlocsubcatcrit}, \Cref{prop:sigmamonog}, and \Cref{prop:knresolution}, we get 
\[
\Stable_{\Sigma(n)} = \Loc(K(n)_*) \subseteq \Loc(\Sigma(n)) \subseteq \cD,
\]  
so $\Stable_{\Sigma(n)} = \cD$. 

In order to prove the second claim, consider a nontrivial thick subcategory $\cT \subseteq \Stable_{\Sigma(n)}^{\omega}$ and write $\cL = \Loc(\cT)$ for the corresponding localizing subcategory of $\Stable_{\Sigma(n)}$. It follows from the first part that $\cL = \Stable_{\Sigma(n)}$ and therefore, by \cite[Thm.~2.1(3)]{neeman_brown}, that $\cT = \Loc(\cT)^{\omega} = \Stable_{\Sigma(n)}^{\omega}$. 
\end{proof}
\begin{quest}
Is it possible to classify the thick subcategories of $\Stable_{\Sigma(n)}$ for $p-1 \mid n$? 
\end{quest}

Let $E_*$ be any height $n$ Landweber exact $BP_*$-algebra. As noted previously these give rise to a category of comodules $(E_*,E_*E)$, and the comodule categories of any two such $BP_*$-algebras are equivalent. Thinking of $E_*$ as the coefficient ring of Morava $E$-theory, the following result gives a lift of \Cref{prop:kngeneric} from Morava $K$-theory to Morava $E$-theory. 
\begin{thm}\label{thm:genericprimes}
\sloppy If $p$ is an $E$-generic prime, that is $p > n+1$, then the localization functor 
$
\xymatrix{\omega\colon\Stable_{E_*E} \ar[r]^-{\sim} & \cD_{E_*E}}
$
is an equivalence of symmetric monoidal stable $\infty$-categories.
\end{thm}
\begin{proof}
We will first show that $E_* \in \Thickid(E_*E)$, the thick tensor ideal in $\cD_{E_*E}$ generated by $E_*E$. To this end, let $E_* \to I^{\bullet}$ be a resolution of $E_*$ by injective $E_*E$-comodules. The assumption that $n < p-1$ implies that there exists some $N\ge 0$ such that $\Ext_{E_*E}^s(E_*,E_*)=0$ for all $s>N$, see the proof of Thm.~10.9 of \cite{ravenel_localization}. Induction on $k$ then shows that $N$ can be chosen large enough so that $\Ext_{E_*E}^s(E_*/I_k,E_*)=0$ for all $s>N$ and all $0 \le k \le n$ as well. Since every dualizable discrete comodule $P$ is finitely presented and projective by \Cref{prop:comodbasics}, it thus follows from the Landweber filtration theorem \cite[Thm.~D]{hs_leht} and the long exact sequence in Ext that 
\begin{equation}\label{eq:genericextvanishing}
\Ext_{E_*E}^s(P,E_*) = 0
\end{equation} 
for all $s>N$. Now consider the exact sequence
\begin{equation}\label{eq:shortalgresolution}
\xymatrix{0\ar[r] & E_* \ar[r]^-{f^0} & I^0 \ar[r]^-{f^1} &I^1 \ar[r]^-{f^2} & \ldots \ar[r]^-{f^N} & I^N \ar[r]^-{g} & \coker(f^N) \ar[r] & 0.}
\end{equation}
\sloppy Recall from \Cref{prop:adamsha} that the dualizable discrete comodules generate $\Comod_{E_*E}$, so \eqref{eq:genericextvanishing} forces the map $g$ to be split. Therefore, $\coker(f^N)$ is a retract of an injective comodule and hence itself injective. But every injective comodule is a retract of an extended comodule by \Cref{lem:injcomodules}, so the resolution \eqref{eq:shortalgresolution} is spliced together from short exact sequences involving only extended comodules. Because short exact sequences induce fiber sequences in $\cD_{E_*E}$, this yields $E_* \in \Thickid(E_*E)$.

To finish the argument, recall from \cref{prop:stableprops} that $\omega\colon \Stable_{E_*E} \to \cD_{E_*E}$ is the localization with respect to the homology isomorphisms. Since $\Stable_{E_*E}$ is stable, it suffices to show that any $M \in \Stable_{E_*E}$ with $H_*M = 0$ must be trivial. Suppose $M \in \Stable_{E_*E}$ with $H_*M = 0$. Define a full subcategory $\cC(M) \subseteq \Stable_{E_*E}$ consisting of those $X \in \Stable_{BP_*BP}$ with $\pi_*(X \otimes M) = 0$; note that $\cC(M)$ is a thick tensor ideal. Since $E_*E$ represents homology, i.e., there is a natural equivalence $H_*(-) \cong \pi_*(E_*E \otimes -)$, we get $E_*E \in \cC(M)$, hence 
\[
E_* \in \Thickid(E_*E) \subseteq \cC(M).
\]
This means that $\pi_*M =0$, thus $M\simeq 0$, and the claim follows. 
\end{proof}

\begin{rem}
More conceptually, the fact that $E_*$ is contained in the thick tensor ideal generated by $E_*E$ is equivalent to the morphism $E_* \to E_*E$ being descendable in the language of \cite{mathew_galois}. The latter statement, in turn, can be shown to be equivalent to the existence of a horizontal vanishing line in the (collapsing) Adams spectral sequence, see \cite[Sec.~4]{mathew_galois}, i.e., to the finite cohomological dimension of $E_* \in \Comod_{E_*E}$. 
\end{rem}

\section{The nilpotence theorem}\label{sec:nilpotence}
In this section we present an algebraic version of the nilpotence theorem in $\Stable_{BP_*BP}$. Our results are not as strong as the nilpotence theorem in stable homotopy theory given by Devinatz, Hopkins, and Smith \cite{nilpotence1,nilpotence2}, principally due to the fact that the detecting family we use does not consist of field objects in $\Stable_{BP_*BP}$. 
\subsection{Equivalent statements of the algebraic nilpotence theorem}
In \cite{nilpotence2} Hopkins and Smith prove that the Morava $K$-theories $K(n)$ can be used to detect nilpotence: a map $f \colon F \to X$ from a finite spectrum to a $p$-local spectrum $X$ is smash nilpotent, i.e., $f^{(m)} = 0$ for some $m \gg 0$, if and only if $K(n)_*f = 0$ for all $0 \le n \le \infty$.

In this section we prove a $\Stable_{BP_*BP}$ variant of this. Our results are more like the nilpotence theorems given in \cite[Sec.~5]{hps_axiomatic}, although we note that they do not follow automatically from their work, since (5.1.2) of \emph{loc.~cit.} is not satisfied in our case. 

Recall that, for $0 \le n \le \infty$, $I_n$ denotes the ideal $(p,v_1,\ldots,v_{n-1}) \subset BP_*$ (with the convention that $I_0 = (0)$); by \cite{land_ii} these are the only invariant prime ideals in $BP_*$. In analogy with the notation for a type $n$ complex in stable homotopy theory, we let $F(n)_* = BP_*/I_n$.\footnote{This often appears in the literature as $P(n)_*$.} We then define $\Tel(n)_*$ as the localization $v_n^{-1}F(n)_*$ (by convention we set $\Tel(0)_* = \mathbb{Q}$ and  $\Tel(\infty)_* = \F_p$). These play the role of the detecting theories in this context (see also \Cref{thm:moravak} below).

Our version of the nilpotence theorem takes the following form.  
\begin{thm}\label{thm:algnilpotence1}(Algebraic nilpotence Theorem I - weak version)
	\begin{enumerate}
		\item  Suppose $F,X \in \Stable_{BP_*BP}$ with $F$ compact, then a map $f\colon F \to X$ is smash nilpotent, i.e., $f^{(m)}=0$ for some $m \gg 0$, if $\Tel(n)_*\otimes_{BP_*}f = 0$ for all $0 \le n \le \infty$. 
		\item A self map $f \colon \Sigma^i F \to F$ for $F \in \Stable_{BP_*BP}^\omega$ is nilpotent, in the sense that $f^j \colon \Sigma^{ij} F \to F$ is null for some $j \gg 0$, if and only if $\Tel(n)_*\otimes_{BP_*}f$ is nilpotent for all $0 \le n \le \infty$. 
	\end{enumerate}
\end{thm}
\begin{thm}\label{thm:algnilpotence2}(Algebraic nilpotence Theorem II - weak version)
	\begin{enumerate}
		\item Suppose $X \in \Stable_{BP_*BP}$, then a map $f\colon BP_* \to X$ is smash nilpotent, i.e., $f^{(m)}=0$ for some $m \gg 0$,  if $\pi_*(\Tel(n)_*\otimes_{BP_*}f) = 0$ for all $0 \le n \le \infty$. 
		\item Let $R$ be a ring object in $\Stable_{BP_*BP}$. Then an element $\alpha \in \pi_*R \cong \Ext_{BP_*BP}(BP_*,R)$ is nilpotent if and only if $\pi_*(\Tel(n)_* \otimes_{BP_*}\alpha)$ is nilpotent for all $0 \le n \le \infty$. 
\end{enumerate}
\end{thm}
	We will prove these below in \Cref{sec:nilpotencepf}. Our proof follows closely the ideas of the original proof of Hopkins and Smith, and we start by building a Bousfield decomposition similar to that seen in the ordinary stable homotopy category.
	
\begin{rem}
We refer to these theorems as weak versions of the nilpotence theorem because they do not account for all periodic elements in $\pi_*BP_*$, but only those of Adams--Novikov filtration 0, which in turn correspond to the classical periodic elements $v_n$ in ordinary stable homotopy theory. This manifests itself in the fact that the telescopes $\Tel(n)_*$ are not field objects, and we thus cannot deduce a description of the Balmer spectrum of $\Stable_{BP_*BP}$. 

However, in forthcoming work with A.~Krause, we will study the global structure of $\Stable_{BP_*BP}$ in more detail. In particular, we will establish a much more refined description of the thick subcategories of compact objects by constructing a more sophisticated detecting family. 
\end{rem} 

\subsection{The proof of the algebraic nilpotence theorem}\label{sec:nilpotencepf}
We start by recalling the basic definition of a Bousfield class, specialized to the category $\Stable_{BP_*BP}$.
\begin{defn}
	Let $M,N \in \Stable_{BP_*BP}$. We say that $M$ and $N$ are Bousfield equivalent if, given any $X \in \Stable_{BP_*BP}$, we have $M \otimes_{BP_*} X \simeq 0$ if and only if $N \otimes_{BP_*} X \simeq 0$. We write $\bc{M}$ for the Bousfield class of $M$.  
\end{defn}
\begin{lem}\label{lem:pinftyhfp}
There is an equivalence $\colim_m F(m)_* \simeq \F_p$ in $\Stable_{BP_*BP}$
\end{lem}
\begin{proof}
There is a natural map $\colim_m F(m)_* \to \F_p$. Since this map is in $\Stable_{BP_*BP}^{\le 0}$, it suffices to check that it is a quasi-isomorphism, which is clear: Indeed, this map is even an isomorphism in $\Comod_{BP_*BP}$. 
\end{proof}
Recall that we denote $\Tel(n)_* = v_n^{-1}F(n)_*$. 
\begin{lem}\label{cor:bcbp}
For any $m \ge 0$, we have an identity of Bousfield classes
\[
\bc{BP_*} = \bc{F(m+1)_*} \oplus \bigoplus_{i=0}^m \bc{\Tel(i)_*}.
\]
\end{lem}
\begin{proof}
By virtue of the general formula $\bc{M} = \bc{M/v} \oplus \bc{v^{-1}M}$ for any self-map $v\colon \Sigma^dM \to M$, see \cite[Prop.~3.6.9(d)]{hps_axiomatic} or \cite[Lem.~1.34]{ravenel_localization}, we see that $\bc{F(m)_*} = \bc{F(m+1)_*} \oplus \bc{\Tel(m)_*}$.  Since $F(0)_* = BP_*$, the result then follows inductively. 
\end{proof}
\begin{rem}
	This result also appears in the proof of Lemma 4.10 of \cite{hovey_chromatic}. 
\end{rem}

For the remainder of this subsection, we will omit all suspensions from the notation. Let $f \colon BP_* \to X$ be a map in $\Stable_{BP_*BP}$. We write 
\[
T_f = \colim(BP_* \xrightarrow{f} X\simeq BP_*\otimes X \xrightarrow{f \otimes 1} X \otimes X \xrightarrow{f \otimes 1 \otimes 1} \ldots)
\]
for the corresponding telescope and $f^{(\infty)}\colon BP_* \to T_f$ for the canonical map. 
\begin{lem}\label{lem:nilptelescope}
Let $R \in \Stable_{BP_*BP}$ be a ring object with unit $\iota\colon BP_* \to R$ and $f\colon BP_* \to X$ some map in $\Stable_{BP_*BP}$. The following statements are equivalent:
	\begin{enumerate}
		\item $R \otimes T_f \simeq 0$.
		\item $\iota \otimes f^{(\infty)}\colon BP_* \to R \otimes T_f$ is zero.
		\item $\iota \otimes f^{(n)}\colon BP_* \to R \otimes X^{(n)}$ is zero for $n \gg 0$. 
		\item $1_{R} \otimes f^{(n)}\colon R \to R \otimes X^{(n)}$ is zero for $n\gg 0$. 
	\end{enumerate}
\end{lem}
\begin{proof}
This is proven as in \cite[Lem.~2.4]{nilpotence2}.
\end{proof}
\begin{rem}
	The proof of this result uses the compactness of $BP_*$, and hence it is cruical that we work in $\Stable_{BP_*BP}$, and not just $\cD_{BP_*BP}$. 
\end{rem}

We now prove the first algebraic nilpotence theorem. 
\begin{proof}[Proof of \Cref{thm:algnilpotence1}]
The `only if' direction of Part (2) is clear, and we note that the other direction follows from Part (1). Indeed, this is the same argument as in \cite{nilpotence2}, namely we replace $f \colon \Sigma^i F \to F$ with its adjoint $f^\# \colon \Sigma^i BP_* \to DF \otimes F$. To prove part (1) we can similarly replace $f \colon F \to X$ with its adjoint $f^\# \colon BP_* \to DF \otimes X$, and so reduce to the case\footnote{In particular, the claim follows from \Cref{thm:algnilpotence2}(2). However, the proof of the latter relies on \Cref{thm:algnilpotence1}, so we cannot apply it here.} where $F = BP_*$.

Let $T_f$ be the telescope associated to $f$. We first start by assuming that $1_{\Tel(n)_*} \otimes f = 0$ for all $0 \le n \le \infty$, so that $\Tel(n)_* \otimes T_f \simeq 0$ for all $n$. By \Cref{lem:nilptelescope} we have to show  $BP_* \otimes T_f \simeq T_f \simeq 0$. By the Bousfield decomposition of \Cref{cor:bcbp}, it then suffices to prove that $F(n)_* \otimes T_f \simeq 0$ for $n \gg 0$. Using \Cref{lem:nilptelescope} again, this will follow from $BP_* \to F(n)_* \otimes T_f$ being null for sufficiently large $n$. To this end, compactness of $BP_*$ gives a factorization
\[
\xymatrix{BP_* \ar[r] \ar@{-->}[d] & \F_p \otimes T_f\ar[d]^{\sim} \\
F(n) \otimes T_f \ar[r] & \colim_n F(n) \otimes T_f,}
\]
where the right vertical equivalence was established in \Cref{lem:pinftyhfp}. By assumption, the top horizontal map is zero, so the claim follows. 
\end{proof}
The proof of the second nilpotence theorem follows closely the one given in \cite[Thm.~5.1.3]{hps_axiomatic}. 
\begin{proof}[Proof of \Cref{thm:algnilpotence2}]
 Once again it suffices to prove part (1). To see this, consider the commutative diagram
	\[
	\xymatrix{
\Sigma^{ktm}BP_* \ar[r]^-{(\alpha^t)^{\otimes m}} \ar[dr]_{\alpha^{tm}} & R^{\otimes m} \ar[d]^{\mu} \\
& R.
}
	\]
If (1) holds, then $\alpha^{tm}$ is null for $m \gg 0$, so that $\alpha$ is nilpotent. The other direction of (2) is clear. 

	The proof of (1) is identical to \cite[Thm.~5.1.3]{hps_axiomatic}, which we repeat for the convenience of the reader. Namely, $\Tel(n)_*$ is a ring object in $\Stable_{BP_*BP}$, so there exist maps $\eta \colon BP_* \to \Tel(n)_*$ and $\mu \colon \Tel(n)_* \otimes_{BP_*} \Tel(n)_* \to \Tel(n)_*$ satisfying the usual relations. Suppose $f \colon BP_* \to X$ is such that $\pi_*(\Tel(n)_* \otimes_{BP_*}f)$ is zero, so that the composite $
BP_* \xr{\eta} \Tel(n)_* \xr{1 \otimes f} \Tel(n)_* \otimes_{BP_*} X$
	is null. But 
	$1 \otimes f$ factors as $\Tel(n)_* \xr{(1 \otimes f) \circ \eta} \Tel(n)_* \otimes_{BP_*} \Tel(n)  \otimes_{BP_*} X \xr{\mu \otimes 1} \Tel(n)_* \otimes_{BP_*} X$ so that $1 \otimes f$ is null. We now apply \Cref{thm:algnilpotence1}(1). 
\end{proof}

\subsection{Base-change and the algebraic telescope conjecture}
The goal of this subsection is to generalize the main structural results of Section 8 in \cite{bhv1} to the Hopf algebroid $(F(m)_*,F(m)_*F(m))$. In particular, we deduce an algebraic version of the telescope conjecture for $\Stable_{BP_*BP}$, which is analogous to Ravenel's theorem that $L_n^fBP \simeq L_nBP$ for all $n\ge 0$. 

There are Landweber exact $F(m)_*$-algebras $E(m,n)_* = \Tel(m)_*/(v_{n+1},v_{n+2},\ldots)$ for all $m \le n$, giving rise to Hopf algebroids $(E(m,n)_*,E(m,n)_*E(m,n))$. These theories come with a natural base-change functor
\[
\xymatrix{\Phi(m,n)_*\colon \Stable_{F(m)_*F(m)} \ar[r] & \Stable_{E(m,n)_*E(m,n)},}
\]
defined by $\Phi(m,n)_*(M) = E(m,n)_* \otimes_{F(m)_*}M$ for any $M \in \Stable_{F(m)_*F(m)}$. This functor clearly preserves arbitrary colimits and thus admits a right adjoint $\Phi(m,n)^*$. 

For a fixed integer $n \ge m$, we will write $(E_*,E_*E)$ for $(E(m,n)_*,E(m,n)_*E(m,n))$, and similarly $(\Phi_*,\Phi^*)$ for the base-change adjunction just constructed.

\begin{prop}\label{prop:pmplethora}
The functor $\Phi^*\colon \Stable_{E_*E} \to \Stable_{F(m)_*F(m)}$ is bimonadic, in the sense that it satisfies the following properties:
\begin{enumerate}
	\item $\Phi^*$ has a left adjoint $\Phi_*$.
	\item $\Phi^*$ has a right adjoint $\Phi_!$.
	\item The counit map $\Phi_*\Phi^*\to \Id$ is an equivalence, so $\Phi^*$ is conservative. 
\end{enumerate}
In particular, the pairs $(\Phi_* \dashv \Phi^*)$ and $(\Phi^*\dashv \Phi_!)$ are monadic and comonadic, respectively.
\end{prop}
\begin{proof}
	The proof is the same as \cite[Prop.~8.13]{bhv1}. 
\end{proof}

Let $m \le k$ and consider the compact object $F(m)_*/I_k\in \Stable_{F(m)_*F(m)}$. In the terminology of \cite{bhv1}, the pair $(\Stable_{F(m)_*F(m)},F(m)_*/I_k)$ forms a local duality context. 

\begin{thm}\label{thm:pmstablebasechange}
Let $E_*$ be a Landweber exact $F(m)_*$-algebra of height $n\ge m$, then the ring map $F(m)_* \to E_*$ induces a natural equivalence
\[
\xymatrix{\Stable_{F(m)_*F(m)}^{I_k-\mathrm{loc}} \ar[r]^-{\sim} & \Stable_{E_*E}^{I_k-\mathrm{loc}}}
\]
for any $m \le k\le n+1$. 
\end{thm}
\begin{proof}
		The same as \cite[Thm.~8.19]{bhv1}. 
\end{proof}
\begin{cor}\label{cor:ethybasechange}
For any $n$ and $m \le k \le n+1$, there is a natural equivalence of stable categories
\[
\xymatrix{\Stable_{E_*E}^{I_{k}-\mathrm{loc}} \ar[r]^-{\sim} & \Stable_{F_*F}}
\] 
for any Landweber exact $F(m)_*$-algebra $F_*$ of height $k-1$.
Furthermore, there is a natural equivalence $\Phi_*L_{I_k}^{F(m)} \simeq L_{I_{k}}^E\Phi_*$, i.e., the following diagram commutes:
\[
\begin{xymatrix}
{
	\Stable_{F(m)_*F(m)}	\ar[d]_-{\Phi_*}\ar^-{L_{I_k}^{F(m)}}[r] & \Stable_{F(m)_*F(m)}^{I_k-\mathrm{loc}} \ar[d]^{\Phi_*}\\
	\Stable_{E_*E} \ar_-{L_{I_k}^{E}}[r] & \Stable_{E_*E.}^{I_k-\mathrm{loc}} 
}	
\end{xymatrix}
\]
\end{cor}
\begin{proof}
	The first claims follows since $\Stable_{E_*E}^{I_k-\mathrm{loc}} \simeq \Stable_{F_*F}$, while for the second one argues as in \cite[Cor.~8.23]{bhv1}. 
\end{proof}

Recall that the telescope conjecture in stable homotopy theory is equivalent to the statement that there is an equivalence of Bousfield classes of spectra $\bc{\Tel(m)} = \bc{K(m)}$ \cite{hovey_csc}, where $\bc{\Tel(m)}$ denotes the Bousfield class of the telescope of a finite spectrum of type $m$. 

	One can ask the same question here: Is $\bc{\Tel(m)_*} = \bc{K(m)_*}$?\footnote{However, note that $\Tel(m)_*$ is not isomorphic to $\pi_*\Tel(m)$, but rather $BP_*\Tel(m)$ in case the corresponding Smith--Toda complex exists.} The main problem in asking this question is that $K(m)_*$ is not an object of $\Stable_{BP_*BP}$. If one modifies the definition of Bousfield class so that $-\otimes_{BP_*}-$ refers to the tensor product in the derived category of $BP_*$-modules only, then one can show that the algebraic telescope conjecture holds. 

\begin{thm}[The algebraic telescope conjecture]\label{thm:moravak}With the above definition, there is an identity of Bousfield classes $\bc{K(m)_*} = \bc{\Tel(m)_*}$ for all $m\ge 0$. 
 
\end{thm}
\begin{proof}
	Both $\Tel(m)_*$ and $K(m)_*$ are Landweber exact $F(m)_*$-algebras, so \Cref{cor:ethybasechange} for $m = k= n+1$ provides a commutative diagram
\[
\xymatrix{\Stable_{F(m)_*F(m)} \ar[d]_{\Phi_*} \ar[rd]^{\Phi_*} & \\
\Stable_{v_{m}^{-1}F(m)_*F(m)} \ar[r]^-{\sim} & \Stable_{K(m)_*K(m)}.}
\]
It follows that $\Tel(m)_* \otimes_{F(m)_*}N= 0$ if and only if $K(m)_* \otimes_{F(m)_*}N =0$ for all $N \in \Stable_{F(m)_*F(m)}$. Because $K(m)_* \otimes_{BP_*} M \simeq K(m)_* \otimes_{F(m)_*} F(m)_* \otimes_{BP_*} M$ for any $M \in \Stable_{BP_*BP}$, we do indeed have the claimed equivalence $\bc{\Tel(m)_*} = \bc{K(m)_*}$. 
\end{proof}
\begin{rem}
	For an alternative formulation of the algebraic telescope conjecture, see \Cref{rem:algtelescoperev}.
\end{rem}

\begin{rem}\sloppy
There is an algebraic analog of Freyd's generating hypothesis~\cite{freyd_stablehomotopy} for $\Stable_{BP_*BP}$; to wit, the algebraic generating hypothesis asks whether the functor
\[
\xymatrix{\pi_*\colon \Stable_{BP_*BP}^{\omega} \to \Mod_{\pi_*BP_*}}
\]
is faithful. As a special case of Lockridge's result \cite[Prop.~2.2.1]{lockridge_thesis}, the algebraic generating hypothesis holds if and only if the $E_2$-page of the Adams--Novikov spectral sequence for the sphere, $\pi_*BP_*$, is totally incoherent as a ring. However, we are not aware of any results about the ring structure of $\pi_*BP_*$. 

Similarly, one can consider the local algebraic generating hypothesis as in \cite{ars}. Using an algebraic version of Brown--Comenetz duality, we suspect that this local version fails for all positive heights, but we will leave the details to the interested reader. 
\end{rem}

\section{Local duality and chromatic splitting for $\Stable_{BP_*BP}$}\label{sec:localduality}

In this section, we introduce the algebraic analogs of Bousfield localization at Morava $K$-theories and Morava $E$-theories, which play a fundamental role in chromatic homotopy theory. Combined with the local duality theory developed in \cite{bhv1}, this provides a convenient framework in which we can study the local structure of $\Stable_{BP_*BP}$. As one instance of this, we discuss an algebraic version of the chromatic splitting conjecture. 

\subsection{Local cohomology and local homology at height $n$}\label{sec:localcohom}
We begin with some recollections from \cite[Sec.~8]{bhv1}. Recall from he previous section that, for $0 \le n < \infty$, we let $I_n$ denote the ideal $(p,v_1,\ldots,v_{n-1}) \subset BP_*$; in particular, $I_0 = (0)$. If $n=\infty$, we define $I_{\infty} = (p,v_1,\ldots) = \bigcup_n I_n$. We refer to \cite[Sec.~2]{bhv1} for background material on localization and colocalization functors. 

\begin{defn}\label{defn:localbphtn}
Let $\Stable_{BP_*BP}^{I_{n+1}-\mathrm{tors}}$ be the localizing subcategory of $\Stable_{BP_*BP}$ generated by $BP_*/I_{n+1}$. The associated colocalization and localization functors will be denoted by $\Gamma_{n}$ and $L_{n}$, respectively.
\end{defn}
We can represent the categories and functors constructed via the following diagram:
\begin{equation}\label{eq:derivedcomod2}
\xymatrix{& \Stable_{BP_*BP}^{I_{n+1}-\textrm{loc}} \ar@<0.5ex>[d] \ar@{-->}@/^1.5pc/[ddr] \\
& \Stable_{BP_*BP} \ar@<0.5ex>[u]^{L_{n}} \ar@<0.5ex>[ld]^{\Gamma_{n}} \ar@<0.5ex>[rd]^{\Lambda^{n}} \\
\Stable_{BP_*BP}^{I_{n+1}-\textrm{tors}} \ar@<0.5ex>[ru] \ar[rr]_{\sim} \ar@{-->}@/^1.5pc/[ruu] & & \Stable_{BP_*BP}^{I_{n+1}-\textrm{comp}}. \ar@<0.5ex>[lu]}
\end{equation}

\begin{rem}
	In \cite[Thm.~2.21]{bhv1} we denoted $\Gamma_n$ and $L_n$ by $\Gamma_{I_{n+1}}$ and $L_{I_{n+1}}$. In order to emphasize the structural similarity with the stable homotopy category and as no confusion is likely to arise, we have changed the notation to $\Gamma_n$ and $L_n$. 
\end{rem}

For all $n \ge 0$ we have morphisms $\Phi \colon BP_* \to E(n)_*$, which give rise to adjoint pairs 
\[
\xymatrix{\Phi_*=\Phi(n)_* \colon \Stable_{BP_*BP} \ar@<0.5ex>[r] & \ar@<0.5ex>[l]\Stable_{E(n)_*E(n)} \colon \Phi(n)^* = \Phi^*.}
\]
The next result summarizes some of the main results of \cite[Sec.~8]{bhv1}; cf.~\Cref{thm:pmstablebasechange}.
\begin{thm}\label{thm:bhvbpresults}
Let $n$ be a nonnegative integer. 
	\begin{enumerate}
		\item There is a natural equivalence of functors $L_n \xr{\sim} \Phi^\ast\Phi_*$ and $\Phi_\ast\Phi^* \xr{\sim} \Id$. 
		\item For any $k \le n+1$, the maps $BP_* \to E(n)_* \to v_{k-1}^{-1}E(n)_*$ induce symmetric moniodal equivalences
		\[
\xymatrix{\Stable_{BP_*BP}^{I_{k}-\text{loc}} \ar[r]^-{\sim}  & \Stable_{E(n)_*E(n)}^{I_k-\text{loc}} \ar[r]^-{\sim}  & \Stable_{v_{k-1}^{-1}E(n)_*E(n)}} 
		\]
		and there is an equivalence $\Stable_{v_{k-1}^{-1}E(n)_*E(n)} \simeq \Stable_{E(k-1)_*E(k-1)}$.
	\end{enumerate}
\end{thm}

In geometric terms, the localization functor $L_{n}$ corresponds to the restriction of a sheaf to the open substack of $\mathcal{M}_{fg}$ of formal groups of height at most $n$. The second part of \Cref{thm:bhvbpresults} can thus be interpreted as giving a presentation of this open substack in terms of the Johnson--Wilson theories $E(n)$.

The inclusions $\Loc(BP_*/I_{n+1}) \subset \Loc(BP_*/I_n)$ give rise to an algebraic chromatic tower
\begin{equation}\label{eq:chromatictower}
\xymatrix{L_{\infty} \ar[r] & \ldots \ar[r] & L_2 \ar[r] & L_1 \ar[r] & L_0,} 
\end{equation}
which is the algebraic analog of the chromatic tower in stable homotopy theory. Our next goal is to study the layers of this tower in more detail. Recall that we can inductively construct objects $BP_*/I_{n}^\infty \in \Stable_{BP_*BP}$ for $n \ge 0$ via cofiber sequences
\begin{equation}\label{eq:fibseq}
BP_*/I_n^{\infty} \to v_{n}^{-1}BP_*/I_{n}^{\infty} \to BP_*/I_{n+1}^\infty,
\end{equation}
under the usual convention $v_0 = p$.

\begin{prop}\label{prop:localcohomcomputations}
For any $M \in \Stable_{BP_*BP}$, we have $L_{n}M \simeq M \otimes  L_n BP_*$, and $L_nBP_*$ can be computed inductively by $L_0BP_* \simeq p^{-1}BP_*$ and cofiber sequences
\[
\Sigma^{-(n+1)}v_{n+1}^{-1}BP_*/I_{n+1}^{\infty} \to L_{n+1}BP_* \to L_nBP_*
\]
for all $n \ge 0$. 
\end{prop}
\begin{proof}
	The first statement is just that $L_n$ is smashing, which follows from the fact that $L_n$ is a finite localization, see \cite[Lem.~3.3.1]{hps_axiomatic}. By \cite[Cor.~8.9]{bhv1} we have $\Gamma_0BP_* \simeq \Sigma^{-1} BP_*/p^\infty$. By definition, $L_0BP_*$ fits in a cofiber sequence $\Gamma_0BP_* \to BP_* \to L_0BP_*$. Comparison with \eqref{eq:fibseq} shows that $L_0BP_* \simeq p^{-1}BP_*$ as claimed. In order to prove the final claim, consider the following commutative diagram
	\[
	\xymatrix{\fib(g_n) \ar[r] \ar[d] & \Gamma_{n+1}BP_* \ar[r]^-{g_n} \ar[d] & \Gamma_nBP_* \ar[d] \\
	0 \ar[r] \ar[d] & BP_* \ar[r]^{\simeq} \ar[d] & BP_* \ar[d] \\
	\fib(l_n) \ar[r] & L_{n+1}BP_* \ar[r]_-{l_n} & L_nBP_*}
	\]
in which all rows and columns are cofiber sequences. The fiber of $g_n$ can be identified with $\Sigma^{-(n+2)}v_{n+1}^{-1}BP_*/I_{n+1}^{\infty}$ by \cite[Cor.~8.9]{bhv1} and \eqref{eq:fibseq}. Therefore, $\fib(l_n) \simeq \Sigma^{-(n+1)}v_{n+1}^{-1}BP_*/I_{n+1}^{\infty}$ and the claim follows. 
\end{proof}

As is standard, we denote the fiber of $L_{n}M \to L_{n-1}M$ by $M_nX$, and call this the $n$th (algebraic) monochromatic layer. 
\begin{cor}\label{cor:mn}
	The $n$th monochromatic layer satisfies the formula $M_nBP_* \simeq \Sigma^{-n} v_n^{-1}BP_*/I_n^\infty$ and is smashing, i.e., for any $X \in \Stable_{BP_*BP}$ there is an equivalence 
	\[
	M_nX \simeq \Sigma^{-n} v_n^{-1}BP_*/I_n^\infty \otimes X.
	\]
	In particular, $M_nE_* \simeq \Sigma^{-n}E_*/I_n^\infty$.  
\end{cor}
\begin{proof}
Since $M_n$ is a fiber of smashing functors, it is smashing as well. The stated formula follows directly from \Cref{prop:localcohomcomputations}, and the rest is clear. 
\end{proof}

\begin{rem}[The algebraic telescope conjecture revisited]\label{rem:algtelescoperev}
	With the introduction of the functor $L_n$ we can give another version of the algebraic telescope conjecture \Cref{thm:moravak}. Recall that in stable homotopy an equivalent formulation of the telescope conjecture is that finite localization with respect to a finite type $n$-spectrum, denoted $L_n^f$, is equivalent to Bousfield localization with respect to $E(n)$, denoted $L_n$, see \cite{mahowaldsadofsky} or \cite{hovey_csc}. Here we formulate an algebraic version of this conjecture. 

	We say that $X \in \Stable_{BP_*BP}$ is $E(n)_*$-local if, for any $T \in \Stable_{BP_*BP}$ with $E(n)_* \otimes_{BP_*}T \simeq 0$, the space of maps $\Hom_{{BP_*BP}}(T,X)$ is contractible. These form a colocalizing subcategory of $\Stable_{BP_*BP}$ and by the $\infty$-categorical version of Bousfield localization \cite[Sec.~5.5.4]{htt} the inclusion of this full subcategory has a left adjoint, which we denote by $L_{E(n)_*}$. An alternative formulation of the algebraic telescope conjecture is that $L_n \simeq L_{E(n)_*}$. 

	It follows from \Cref{thm:bhvbpresults} that $E(n)_* \otimes_{BP_*}X \simeq 0$ if and only if $L_nX \simeq L_nBP_* \otimes_{BP_*}X \simeq 0$, or equivalently that $\bc{L_nBP_*} = \bc{E(n)_*}$. But by \cite[Cor.~11]{miller_finiteloc}	$L_n$ is Bousfield localization with respect to $L_nBP_*$, and hence $L_n \simeq L_{E(n)_*}$. It follows that this version of the algebraic telescope conjecture holds in $\Stable_{BP_*BP}$. 

	In \cite{hovey_chromatic} Hovey considers yet another version of the algebraic splitting conjecture, comparing $L_n$ with the functor given by Bousfield localization at the homology theory corresponding to $E(n)_*$. By \cite[Prop.~3.11]{hovey_chromatic} this functor is given by $H_*(E(n)_* \otimes_{BP_*} X)$ for $X \in \Stable_{BP_*BP}$. Hovey proves that this cannot agree in general with $L_n$, as the former has essential image $\cD_{BP_*BP}$, while the latter has essential image $\Stable_{BP_*BP}$. Nonetheless, the proof of \Cref{thm:genericprimes} shows that when $n < p-1$, so that $\Stable_{E_*E} \simeq \cD_{E_*E}$, these two localizations do agree.
\end{rem}

Similar to the case of $BP_*BP$ above, we can consider the localizing subcategory of $\Stable_{E(n)_*E(n)}$ generated by $E(n)_*/I_n$. There is an associated localization functor $L_{n-1}^E$ which by \cite[Cor.~8.23]{bhv1} has the property that $\Phi_*L_{n-1} \simeq L^E_{n-1}\Phi_*$.  We let $\Delta_n$ denote the functor that is right adjoint to $L_{n-1}^E$, viewed as endofunctors of $\Stable_{E(n)_*E(n)}$, which exists by \cite[Thm.~2.21]{bhv1}; in particular, there is a local duality equivalence 
\begin{equation}\label{eq:adjointdelta}
\iHom_{E(n)_*E(n)}(L_{n-1}^E E(n)_*,M) \simeq \Delta_nM
\end{equation}
for $M \in \Stable_{E(n)_*E(n)}$.

\subsection{Local cohomology at height $\infty$}

We now turn to the height $\infty$ analog of the theory presented above. 
\begin{defn}
Let $\Stable_{BP_*BP}^{I_{\infty}-\mathrm{tors}}$ be the localizing subcategory of $\Stable_{BP_*BP}$ generated by $BP_*/I_{\infty}  \cong \Z/p$. The associated colocalization and localization functors will be denoted by $\Gamma_{\infty}$ and $L_{\infty}$, respectively. 
\end{defn}

Note that $BP_*/I_{\infty} \in \Stable_{BP_*BP}$ is not compact, but we still have a diagram of adjunctions
\[
\xymatrix{\Stable_{BP_*BP}^{I_{\infty}-\mathrm{tors}} \ar@<0.5ex>[r]^-{\iota_{\infty}} & \Stable_{BP_*BP} \ar@<0.5ex>[l]^-{\Gamma_{\infty}}  \ar@<0.5ex>[r]^-{L_{\infty}} & \Stable_{BP_*BP}^{I_{\infty}-\mathrm{loc}}, \ar@<0.5ex>[l]}
\]
where the left adjoints are displayed on top. Recall from \eqref{eq:chromatictower} the algebraic chromatic tower
\[
\xymatrix{\ldots \ar[r] & L_2 \ar[r] & L_1 \ar[r] & L_0}.
\]
The next result identifies the limit of this tower.

\begin{prop}\label{prop:linftycc}
There is a natural equivalence of functors $\xymatrix{L_{\infty} \ar[r]^-{\sim} & \lim_nL_n}$. 
\end{prop}
\begin{proof}
First we note that $BP_*/I_{\infty} \simeq \colim_n BP_*/I_n$, where the colimit is taking along the canonical quotient maps. The inclusion $\Loc(BP_*/I_{\infty}) \subseteq \Loc(BP_*/I_n)$ induces natural transformations $L_{\infty} \to L_n$ for all $n$. Therefore, we have a natural morphism of cofiber sequences of functors
\[
\xymatrix{\Gamma_{\infty} \ar[r] \ar[d]_{\phi} & \Id \ar[r] \ar[d]_{\sim} & L_{\infty} \ar[d] \\
\lim_n\Gamma_n \ar[r] & \Id \ar[r] & \lim_nL_n,}
\]
so it suffices to show that $\phi$ is an equivalence. We will show that $\lim_n\Gamma_n$ is right adjoint to the inclusion functor $\iota_{\infty}$. To this end, let $M \in \Stable_{BP_*BP}^{I_{\infty}-\mathrm{tors}}$ and $N \in \Stable_{BP_*BP}$; we get
\begin{align*}
\Hom(M,\lim_n\Gamma_nN) & \simeq \lim_n\Hom(M,\Gamma_nN) \\
& \simeq \lim_n\Hom(\iota_nM,N) \\
& \simeq \Hom(\colim_n\iota_nM,N) \\
& \simeq \Hom(\iota_{\infty}M,N),
\end{align*}
where the last equivalence from the construction, using the commutative triangle
\[
\xymatrixrowsep{2pc}
\xymatrix{& \Stable_{BP_*BP} \\
\Loc(BP_*/I_{\infty}) = \Stable_{BP_*BP}^{I_{\infty}-\mathrm{tors}} \ar[ru]^{\iota_{\infty}} \ar@{^{(}->}[rr] & & \Stable_{BP_*BP}^{I_{n}-\mathrm{tors}}= \Loc(BP_*/I_{n}). \ar[lu]_{\iota_{n}}}
\]
The claim follows.
\end{proof}

We will see in \Cref{sec:chromaticconvergence} that $L_{\infty}$ is equivalent to the identity functor on a large subcategory of $\Stable_{BP_*BP}$, i.e., we will prove an algebraic version of the chromatic convergence theorem of Hopkins and Ravenel \cite{orangebook}.

\subsection{The algebraic chromatic splitting conjecture}

The goal of this section is to explore an algebraic version of Hopkins's chromatic splitting conjecture for $\Stable_{BP_*BP}$. To this end, we recall that we let $F(n)_*$ denote the quotient $BP_*/I_n$. Note that since $L_n$ is smashing, $L_nF(n)_*$ is a compact object of $\Stable_{BP_*BP}^{I_{n+1}-\text{loc}}$. 

\begin{defn}
\sloppy We define the functor $L_{K(n)}$ to be the composite $\Lambda^{L_nF(n)_*}L_n$, where $\Lambda^{L_nF(n)_*}$ is the completion functor associated to the local duality context $(\Stable_{BP_*BP}^{I_{n+1}-\text{loc}},L_nF(n)_*)$. This definition makes sense because $L_n$ takes essential image in $\Stable_{BP_*BP}^{I_{n+1}-\text{loc}}$.
\end{defn}
Of course, $L_nF(n)_*$ is also an object of $\Stable_{BP_*BP}$ via the canonical inclusion, and we have the following. 
\begin{lem}
$L_{K(n)}$ is Bousfield localization on $\Stable_{BP_*BP}$ with respect to the theory $L_nF(n)_*$.
\end{lem}

\begin{proof}
The argument is similar to the proof of \cite[Prop.~2.31]{bhv1}. It is easy to verify that $L_{K(n)} = \Lambda^{L_nF(n)_*}L_n$ is a localization functor, so it suffices to identify the corresponding category of acyclics. For $X \in \Stable_{BP_*BP}$, we have $L_{K(n)}X \simeq 0$ if and only if $L_nF(n)_* \otimes_{BP_*} L_nX \simeq 0$, because $L_nF(n) \in \Stable_{BP_*BP}^{I_{n+1}-\text{loc}}$ is compact. This in turn is equivalent to $L_nF(n)_* \otimes_{BP_*} X \simeq 0$, and the claim follows.  
\end{proof}

\begin{prop}\label{lem:acfs}
For any $X \in \Stable_{BP_*BP}$ there is a pullback square
\[
\xymatrix{L_nX \ar[r] \ar[d] & L_{K(n)}X \ar[d] \\
L_{n-1}X \ar[r]_-{\iota_X} & L_{n-1}L_{K(n)}X,}
\]
with horizontal fibers equivalent to $\iHom_{BP_*BP}(L_{n-1}BP_*,L_nX) \simeq \Delta_n(\Phi_*X)$.  
\end{prop}
\begin{proof}
	Applying the fracture square \cite[Cor.~2.26]{bhv1} associated to the local duality context $(\Stable^{I_{n+1}-\text{loc}}_{BP_*BP},L_nF(n)_*)$ to $L_nX$ we get a pullback square 
\begin{equation}\label{eq:prefracturesquare}
\xymatrix{L_nX \ar[r] \ar[d] & \Lambda^{L_nF(n)_*}L_nX \ar[d] \\
L_{L_nF(n)_*}L_nX \ar[r]_-{\iota_X} & L_{L_nF(n)_*}\Lambda^{L_nF(n)_*}L_nX}
\end{equation}
for any $X \in \Stable_{BP_*BP}$. Since $L_{K(n)}X = \Lambda^{L_nF(n)_*}L_nX$ by definition, we must show that $L_{L_nF(n)_*}L_nX \simeq L_{n-1}X$. To see this, let us denote by $\cal{M}_n$ the essential image of the functor $\Gamma_{L_nF(n)_*}L_n$ on $\Stable_{BP_*BP}^{I_{n+1}-\text{loc}}$. The same argument as in the first part of \cite[Lem.~7.14]{bhv1} shows that there is a commutative diagram of adjunctions
\[
\xymatrix{\Stable_{BP_*BP}^{I_n-\text{tors}} \ar[d]_{L_n} \ar@<-0.5ex>[r] & \Stable_{BP_*BP} \ar@<-0.5ex>[l]_-{\Gamma_{n-1}} \ar[d]^{L_n}\\
 \cal{M}_n \ar@<-0.5ex>[r]  & \Stable_{BP_*BP}^{I_{n+1}-\text{loc}} \ar@<-0.5ex>[l]_-{\Gamma_{L_nF(n)_*}}.}
 \]
We have a fiber sequence
\[
\xymatrix{\Gamma_{L_nF(n)_*}L_nX \ar[r] & L_nX \ar[r] & L_{L_nF(n)_*}L_nX}
\]
which, using the diagram above, is equivalent to the fiber sequence 
\[
\xymatrix{L_n\Gamma_{n-1}X \ar[r] & L_nX \ar[r] & L_{L_nF(n)_*}L_nX.}
\]
 By comparing with the defining cofiber sequence $\Gamma_{n-1}X \to X \to L_{n-1}X$, we deduce that there are equivalences $L_{L_nF(n)_*}L_nX \simeq L_nL_{n-1}X \simeq L_{n-1}X$. 

To compute the fiber, we work with the fiber sequence associated to the top map in \eqref{eq:prefracturesquare}. By \cite[Thm.~2.21]{bhv1} there is a right adjoint $\Delta_{L_nF(n)_*}^{BP_*}$ to $L_{L_nF(n)_*}$ on $\Stable_{BP_*BP}$, fitting into a fiber sequence
\[
\xymatrix{\Delta_{L_nF(n)_*}^{BP_*} \ar[r] & \text{id} \ar[r] & \Lambda_{BP_*}^{L_nF(n)_*}.}
\]
Moreover, $\Delta$ satisfies the local duality formula $\Delta_{L_nF(n)_*}^{BP_*}(-) \simeq \iHom_{BP_*BP}(L_{L_nF(n)_*}BP_*,-)$. Therefore, the fiber is equivalent to
\begin{align*}
\Delta_{L_nF(n)_*}^{BP_*}(L_nX) & \simeq \iHom_{BP_*BP}(L_{L_nF(n)_*}BP_*,L_nX) \\
& \simeq \iHom_{BP_*BP}(L_{L_nF(n)_*}L_nBP_*,L_nX) \\
& \simeq \iHom_{BP_*BP}(L_{n-1}BP_*,L_nX)
\end{align*}
by the previous paragraph.

For the final equivalence of the statement, note that $\iHom_{BP_*BP}(-,-)$ is equivalent to the internal Hom in $\Stable_{BP_*BP}^{I_{n+1}-\text{loc}}$. Indeed, if $M,N \in \Stable_{BP_*BP}^{I_{n+1}-\text{loc}}$, then 
\[
\begin{split}
\Hom_{{BP_*BP}}(X,\iHom_{BP_*BP}(M,N)) \simeq \Hom_{{BP_*BP}}(X \otimes M,N)
\simeq 0,
\end{split}
\]
for all $X \in \Loc(BP_*/I_{n})$, since $\Stable_{BP_*BP}$ is monogenic and $N \in \Stable_{BP_*BP}^{I_{n+1}-\text{loc}}$. This implies that $\iHom_{BP_*BP}(M,N)$ is $I_{n+1}$-local, from which the claim easily follows. 

By the equivalence of categories of \Cref{thm:bhvbpresults} and using \eqref{eq:adjointdelta}, we thus see that, via the natural inclusion, the fiber in question is equivalent to
\[
\iHom_{E_*E}(\Phi_*L_{n-1}BP_*,\Phi_*L_nX) \simeq \iHom_{E_*E}(L_{n-1}^E E_*,\Phi_*X) \simeq \Delta_n(\Phi_*X),
\]
where we have used the fact that $\Phi_*L_nX \simeq \Phi_*X$, see \Cref{thm:bhvbpresults}(1). 
\end{proof}

The algebraic chromatic fracture square of \Cref{lem:acfs} describes how objects in $\Stable_{BP_*BP}$ are assembled from their local pieces $L_{K(n)}X$. In analogy to Hopkins's chromatic splitting conjecture~\cite[Conj.~4.2]{hovey_csc}, one can ask if the map $\iota_X$ is split for compact $X$ and, if so, how to further decompose its cofiber. 

In fact, there are various versions of the algebraic chromatic splitting conjecture, corresponding to the analogous statements in chromatic homotopy theory. The most conceptual form asks whether  $\iota_X$ is a split monomorphism for any $X \in \Stable_{BP_*BP}^{\omega}$.\footnote{Similar but inequivalent questions have been investigated by Hovey~\cite{hovey_csc} and Devinatz~\cite{devinatz_csc}.} However, we are interested in the more refined statement that also describes the other summand in the splitting. Furthermore, we will focus on the so-called edge case of the algebraic chromatic splitting conjecture corresponding to Hopkins' chromatic splitting conjecture at height $n$ for a type $n-1$ complex. 

To this end, fix $n\ge 0$ and note that the  algebraic chromatic fracture square of \Cref{lem:acfs} remains unchanged when $X$ is localized at $E_*=(E_n)_*$. Therefore, by base-change we may assume without loss of generality that we are working in $\Stable_{E_*E}$ with, and we write $L_n$ for what was previously denoted $L_n^E$.  In \cite[Thm.~6]{devinatzhopkins_subgroups}, Devinatz, Hopkins, and Miller construct a class $\zeta \in \pi_{-1}L_{K(n)}S^0$ by lifting the determinant class $\det \in \Hom_{\mathrm{cts}}(\mathbb{G}_n,E_*)$, the set of continuous functions from the (extended) Morava stabilizer group $\mathbb{G}_n = \mathbb{S}_n \rtimes \operatorname{Gal}(\F_{p^n}/\F_p)$ to $E_*$. If $X$ is a finite spectrum of type $n-1$, then Hopkins's chromatic splitting conjecture stipulates that there is an equivalence
\[
\xymatrix{L_{n-1}X \oplus \Sigma^{-1}L_{n-1}X \ar[r]^-{\sim} & L_{n-1}L_{K(n)}X}
\]
induced by the natural inclusion and $\zeta$. This conjecture is known to hold for $n=1$ as well as $n=2$ and $p\ge 3$, but needs to be modified for $n=2$ and $p=2$ by work of Beaudry \cite{beaudry_csc}. Therefore, we will assume that $p$ is large with respect to $n$ for the remainder of this section. 

In work in progress of the first author with Beaudry and Peterson~\cite{acsc}, we explain how to construct an algebraic class $\zeta \in \Hom_{E_*E}(E_*,\lim_i^1E_*/I_n^i)$ associated to $\det \in E_*^{\vee}E = \pi_*L_{K(n)}(E \otimes E) \cong \Hom_{\mathrm{cts}}(\mathbb{G}_n,E_*)$.\footnote{The skeptical reader may consider the existence of this class as being part of the conjecture throughout this section.} In order to lift this class to an analog in $\Stable_{E_*E}$ of the topological class $\zeta$, we need the following lemma, which was proven in \cite[Thm.~8.31]{bhv1}.

\begin{lem}\label{lem:localhomologyformula}
For a finitely presented $E_*E$-comodule $M$ and any $s\ge 0$, there is a canonical isomorphism $H^sL_{K(n)}M \cong \lim_{i}^sM/I_n^i$ of $E_*E$-comodules.
\end{lem}

Therefore, the convergent hyperext spectral sequence yields a (potentially trivial) class $\zeta \in \Ext_{E_*E}^1(E_*, L_{K(n)}E_*)$, i.e., a map
\[
\xymatrix{\zeta_{E_*}\colon \Sigma^{-1}E_* \ar[r] & L_{K(n)}E_*}
\]
in $\Stable_{E_*E}$. Note that, since $p$ is assumed to be large with respect to $n$, \Cref{thm:genericprimes} implies that $\Stable_{E_*E} \simeq \cD_{E_*E}$. It follows that there are corresponding maps $\zeta_{M}\colon L_{n-1}M \to L_{n-1}L_{K(n)}M$ for any $M \in \Stable_{E_*E}^{\omega}$. We may thus state an algebraic version of the chromatic splitting conjecture. 

\begin{conj}[Algebraic chromatic splitting conjecture]\label{conj:acsc}
For any $M \in \Thick(E_*/I_{n-1})$ there is an equivalence 
\[
\xymatrix{L_{n-1}M \oplus \Sigma^{-1}L_{n-1}M \ar[r]^-{\sim} & L_{n-1}L_{K(n)}M}
\]
induced by the maps $\iota_M$ and $\zeta_M$.
\end{conj}

Note that a thick subcategory argument reduces this conjecture to the case $M = E_*/I_{n-1}$. We will therefore restrict attention to the case that $M$ is a finitely presented $I_{n-1}$-torsion $E_*$-comodule viewed as an object of $\Stable_{E_*E}$ concentrated in degree $0$. There are then two other equivalent formulations of this conjecture, in particular relating it to the version of the algebraic chromatic splitting conjecture proposed in unpublished work by Hopkins and Sadofsky. Combining the following result with~\cite{acsc}, this would show that \Cref{conj:acsc} is equivalent to the topological chromatic splitting conjecture. 

\begin{prop}\label{prop:acsceq}
For a finitely presented $I_{n-1}$-torsion $E_*E$-comodule $M$ the following three statements are equivalent:
\begin{enumerate}
	\item The algebraic chromatic splitting conjecture holds for $M$.
	\item The maps $\iota_M$ and $\zeta_M$ induce isomorphisms 
		\[
		\mathrm{lim}_i^sM/v_{n-1}^i \cong 
			\begin{cases}
				M & \text{if } s=0 \\
				v_{n-1}^{-1}M & \text{if } s=1 \\
				0 & \text{otherwise}.
			\end{cases}
		\]
	\item The class $\zeta_M$ induces an equivalence $D(L_{n-1}E_*) \otimes M \simeq \Sigma^{-2}L_{n-1}M$, where $D$ denotes internal duality in the stable category $\Stable_{E_*E}$. 
\end{enumerate}
\end{prop}
\begin{proof}
Let $M \in \Thick(E_*/I_{n-1})$ and consider the fiber sequence
\[
\xymatrix{\Delta_nM \ar[r] & L_{n-1}M \ar[r]^-{\iota_M} & L_{n-1}L_{K(n)}M,}
\]
which follows from \Cref{lem:acfs} (recall that we assume that $M \in \Stable_{E_*E}$). On the one hand, if the algebraic chromatic splitting conjecture holds for $M$, then we obtain an equivalence 
\[
\Sigma\Delta_nM \simeq \Sigma^{-1}L_{n-1}M.
\]
On the other hand, \eqref{eq:adjointdelta} provides a natural equivalence $\Delta_nM \simeq \iHom(L_{n-1}E_*,M)$, hence $\Delta_nM \simeq D(L_{n-1}E_*) \otimes M$ by compactness of $M$. This shows that (1) implies (3).

Now assume Statement (3), which is equivalent to $\Delta_nM \simeq \Sigma^{-2}L_{n-1}M$ as just shown. From the long exact sequence in cohomology associated to the fiber sequence $\Delta_nM \to M \to L_{K(n)}M$ we thus obtain
\[
H^sL_{K(n)}M \cong
	\begin{cases}
		M & \text{if } s=0 \\
		L_{n-1}M & \text{if } s=1 \\
		0 & \text{otherwise}.
	\end{cases}
\]
The isomorphisms in (2) follow from this by virtue of \Cref{lem:localhomologyformula} and~\Cref{prop:localcohomcomputations}, 
because $M$ is $I_{n-1}$-torsion.

Finally, Condition (2) implies that the map 
\[
\xymatrix{(\iota_M,\zeta_M)\colon L_{n-1}M \oplus \Sigma^{-1}L_{n-1}M \ar[r]^-{\sim} & L_{n-1}L_{K(n)}M}
\]
is a quasi-isomorphism. Since $\Stable_{E_*E} \simeq \cD_{E_*E}$ for large $p$ by \Cref{thm:genericprimes}, this gives the algebraic chromatic splitting conjecture for $M$.
\end{proof}

\begin{rem}
Statement (3) of the previous proposition says in particular that $L_{n-1}(E_*/I_{n-1})$ is reflexive (or weakly dualizable) as an object in the derived category of $(E_*/I_{n-1},E_*E/I_{n-1})$-comodules, i.e., that $L_{n-1}E_*/I_{n-1} \simeq D_{I_{n-1}}^2(L_{n-1}E_*/I_{n-1})$ via the canonical map, where $D_{I_{n-1}} = \iHom_{E_*E/I_{n-1}}(-,E_*/I_{n-1})$. This is remarkable, since $L_{n-1}E_*/I_{n-1} \in \Stable_{E_*E/I_{n-1}}$ is not compact and hence not dualizable. 
\end{rem}

\begin{rem}
There is also a version of \Cref{prop:acsceq} that is independent of the existence of the algebraic analog of $\zeta$. In this case, the proof still gives the implications $(1) \implies (3) \implies (2)$.  
\end{rem}

\section{The algebraic chromatic convergence theorem}\label{sec:chromaticconvergence}

The chromatic convergence theorem shows that a finite spectrum $F$ can be recovered from its chromatic localizations $L_nF$. The goal of this section to establish an algebraic analog of this result for $\Stable_{BP_*BP}$.

\subsection{The theory of algebraic $n$-buds and comodules}
In this section, we present an analog of the parts of the theory of $n$-buds of formal groups as developed by Goerss~\cite[Sec.~3.3]{goerss_quasi-coherent_2008} to the setting of $BP_*BP$-comodules, and then generalize it to $\Stable_{BP_*BP}$. This will provide an appropriate setting for the first version of our algebraic chromatic convergence theorem, see \Cref{thm:chromaticconvergence}.

\begin{defn}
For any $0\le n \le \infty$ let $(B_n,W_n)$ be the Hopf algebroid representing $n$-buds of formal groups. Explicitly, $B_n = \Z_{(p)}[v_1,\ldots,v_n]$ and $W_n = B_n[a_1,\ldots,a_{n-1}]$; viewing $(B_n,W_n)$ as a sub-Hopf algebroid of $(BP_*,BP_*BP)$ via the natural inclusion map
\[
\xymatrix{q_n\colon(B_n,W_n) \ar[r] & (BP_*,BP_*BP)}
\]
determines the structure maps. These functors induce a natural isomorphism
\begin{equation}\label{eq:nbudhopfalgebroid}
\colim_n(B_n,W_n) \cong (BP_*,BP_*BP)
\end{equation}
of Hopf algebroids, which motivates to write $(B_{\infty},W_{\infty}) = (B,W) = (BP_*,BP_*BP)$. 
\end{defn}

The map $q_n$ gives rise to functors of abelian categories
\[
\xymatrix{(q_n)_*\colon \Comod_{W_n} \ar@<0.5ex>[r] & \Comod_{BP_*BP}\colon (q_n)^*, \ar@<0.5ex>[l]} 
\]
where the left adjoint is given by $(q_n)_*M = BP_* \otimes_{B_n}M$ with its natural comodule structure. As $BP_*$ is flat as a $B_n$-module, $q_n$ is exact. Note that in Goerss's algebro-geometric language \cite{goerss_quasi-coherent_2008}, the left adjoint is denoted by $(q_n)^*$, whereas our choice of notation is consistent with the one in \Cref{sec:localduality}. Since the $q_n$ are compatible with each other for varying $n$, we obtain an induced functor
\[
\xymatrix{q_*\colon \colim_n\Comod_{W_n} \ar[r] & \Comod_{BP_*BP},}
\]
where the colimit is taken over the functors $(q_{n,n+1})_*\colon \Comod_{W_n} \to \Comod_{W_{n+1}}$ sending a comodule $M$ to $B_{n+1} \otimes_{B_n}M$. It is proven in \cite[Prop.~3.25]{goerss_quasi-coherent_2008} that $q_*$ is faithful and that it induces an equivalence
\[
\xymatrix{q_*^{\omega}\colon \colim_n \Comod_{W_n}^{\omega} \ar[r]^-{\sim} & \Comod_{BP_*BP}^{\omega},}
\]
see also \cite{smithling}. The next result relates two important properties of a comodule to the categories $\Comod_{W_n}$. Recall that a $BP_*BP$-comodule $M$ is said to have projective $BP_*$-dimension $n\ge 0$ if the underlying $BP_*$-module $\epsilon_*(M)$ has projective dimension $n$. 

\begin{lem}\label{lem:budstorsiondim}
For a comodule $M\in \Comod_{BP_*BP}$, consider the following conditions:
\begin{enumerate}
	\item $M$ is in the essential image of $(q_r)_*$.
	\item The projective $BP_*$-dimension of $M$ is at most $r+1$.
	\item $M$ is $v_{r+2}$-torsion free. Equivalently, $M$ is $v_i$-torsion free for all $i\ge r+2$. 
\end{enumerate}
Then Condition (1) implies the Condition (2). If $M$ is additionally bounded below, then Condition (2) implies Condition (3). 
\end{lem}
\begin{proof}
Suppose first that $M$ is in the essential image of $(q_r)_*$, say $M \cong (q_r)_*N$. Since the homological dimension of $B_r$ is $r+1$, $N$ admits a projective resolution by $B_r$-modules of length at most $r+1$. Since $(q_r)_*$ preserves projective objects, it follows that (1) implies (2). 

As shown in \cite[Prop.~2.5]{johnson_torsion}, a $BP_*BP$-comodule $M$ is $v_{r+1}$-torsion free if and only if it is $v_m$-torsion free for all $m >r$, which gives the last claim in Condition (3). Moreover, Johnson and Yosimura prove that for bounded below $M$, this condition follows from $M$ having homological $BP_*$-dimension $\le r+1$, see \cite[Prop.~3.7]{johnson_torsion}, hence (2) implies (3). 
\end{proof}

In order to prove the algebraic chromatic convergence theorem, we will use a derived version of this theory. 

To this end, let $\Stable_{W_n}$ for $0 \le n < \infty$ denote the stable category associated to $\Comod_{W_n}$.
\begin{lem}
	The stable category $\Stable_{W_n}$ is monogenic for all $n$. 
\end{lem}
\begin{proof}
	By \Cref{lem:lamonogenic} it suffices to show that $(B_n,W_n)$ is a Landweber Hopf algebroid, and by the argument given in \cite[Thm.~6.6]{hovey_htptheory} this will follow if we can show that every finitely presented $W_n$-comodule has a Landweber filtration. The proof for this is similar to that for $BP_*BP$-comodules; in fact, it is simpler because $B_n$ is Noetherian. First, the invariant radical ideals in $W_n$ are given by $I_k \cap B_n$ for $k \le n$ \cite[Ex.~5.10]{hollander}. We then apply \cite[Thm.~3.3]{land_filt} with $R = B_n \cong \Z_{(p)}[v_1,\ldots,v_n], S \cong \Z_{(p)}[a_1,\ldots,a_n]$ (so that $R \otimes S \cong B_n[a_1,\ldots,a_n] = W_n$), and $\Psi \colon B_n \to W_n$ given by the right unit of the Hopf algebroid $(B_n,W_n)$. 
\end{proof}
\begin{prop}\label{prop:budscompact}
The maps $q_n$ introduced above induce an exact functor $q_*\colon\colim_n\Stable_{W_n} \to \Stable_{BP_*BP}$, which restricts to an equivalence
\[
\xymatrix{q_*^{\omega}\colon \colim_n\Stable_{W_n}^{\omega} \ar[r]^-{\sim} & \Stable_{BP_*BP}^{\omega}}
\]
of stable $\infty$-categories. 
\end{prop}
\begin{proof}
For any pair $(m,n)$ with $0 \le m\le n \le \infty$, the map $q_{m,n}\colon (B_m,W_m) \to (B_n,W_n)$ of Hopf algebroids induces a functor
\[
\xymatrix{(q_{m,n})_*\colon \Stable_{W_m} \ar[r] & \Stable_{W_n}}
\]
which preserves colimits and compact objects. These functors are compatible with each other, hence we obtain a commutative diagram
\[
\xymatrix{\colim_n \Stable_{W_n} \ar[r]^-{q_*} \ar[d]_{\sim} & \Stable_{BP_*BP} \ar[d]^{\sim} \\
\colim_n \Mod_{\End_{W_n}(B_n)} \ar[r]_-{q_*} & \Mod_{\End_{BP_*BP}(BP_*)},}
\]
where the vertical equivalences follow from derived Morita theory and the previous lemma, see \cite[Thm.~3.1.1]{schwedeshipley_modules} and \cite[Thm.~7.1.2.1]{ha}. Passing to compact objects and using that the functor $\Mod_{-}^{\omega}\colon \Alg_{\mathbb{E}_{\infty}} \to \mathrm{Cat}_{\infty}$ preserves filtered colimits, as is shown for example in the proof of \cite[Prop.~2.4.1]{ms_picard}, this gives a functor
\[
\xymatrix{q_*^{\omega}\colon \Mod_{\colim_n \End_{W_n}(B_n)}^{\omega} \simeq \colim_n\Mod_{\End_{W_n}(B_n)}^{\omega} \ar[r] & \Mod_{\End_{BP_*BP}(BP_*)}^{\omega}.}
\]
Unraveling the construction, note that $q_*^{\omega}$ is induced by the natural map 
\[
\phi\colon \colim_n \End_{W_n}(B_n) \to \End_{BP_*BP}(BP_*),
\] 
so it suffices to prove that $\phi$ is an equivalence. To this end, let $C^*(B_n)$ be the cobar construction on $B_n$ in $\Comod_{W_n}$. Using \eqref{eq:nbudhopfalgebroid} and exactness of $q_*$, we compute
\begin{align*}
\colim_n\Ext_{W_n}^*(B_n,B_n) & \cong \colim_nH^*(\Hom_{W_n}(B_n,C^*(B_n))) \\
& \cong H^*(\Hom_{\colim_nW_n}(\colim_nB_n,C^*(\colim_nB_n))) \\
& \cong H^*(\Hom_{BP_*BP}(BP_*,C^*(BP_*))) \\
& \cong \Ext_{BP_*BP}^*(BP_*,BP_*),
\end{align*}
hence $\phi$ is an equivalence. 
\end{proof}

\subsection{Chromatic convergence}

Before we can come to the proof of the algebraic chromatic convergence theorem, we need a technical lemma regarding the vanishing of derived functors of inverse limits of comodules. We remind the reader about our grading conventions, see \Cref{sec:conventions}.

\begin{lem}\label{lem:limzerohomolstructuremaps}
Suppose $d\in \Z$ and $M=(M_n,\phi_n)_n \in (\Stable_{BP_*BP}^{\le d})^{\mathbb{N}^{\mathrm{op}}}$ is an inverse system with structure maps $\phi_n\colon M_{n+1} \to M_n$. If for any $q \in \Z$ there exists $m(q)$ such that the induced map $H^q(\phi_n)$ is zero for all $n>m(q)$, then $\lim_nM_n \simeq 0$. 
\end{lem}
\begin{proof}
Since $\Stable_{BP_*BP}^{\le d} \simeq \cD_{BP_*BP}^{\le d}$, it suffices to show that $H^{k}\lim(M_n) = 0$ for all $k$. To this end, note that the convergent hypercohomology spectral sequence takes the form 
\[
E_2^{p,q}\cong \lim{}^pH^q(M_n) \implies \lim{}^{p+q}M_n
\]
where the derived limits on the $E_2$-page are computed with respect to the  structure maps $H^q(\phi_n)$. By assumption, these morphisms are zero for all $n > m(q)$, so it follows from \cite[Lem.~1.11]{jannsen_contetcohom} that $E_2^{p,q} = 0$ for all $p$ and $q$. Therefore, $H^{k}\lim(M_n) \cong \lim^{k}M_n =0$ for all $k \in \Z$. 
\end{proof}

\begin{lem}\label{lem:chromaticconvergencezeromaps}
If $X \in \cD_{B_r}$ for some $r \ge 0$, then the natural map
\[
\xymatrix{H_*\Gamma_{n} BP_*\otimes_{B_r}X \ar[r] & H_*\Gamma_{n-1}BP_*\otimes_{B_r}X}
\]
of $BP_*$-modules is zero for all $n > r$. 
\end{lem}
\begin{proof}
Consider the following segment of the long exact sequence in homology corresponding to the cofiber sequence $BP_*/I_n^{\infty}  \to BP_*/I_n^{\infty}[v_n^{-1}] \to BP_*/I_{n+1}^{\infty}$:
\[
\xymatrix{H_*(BP_*/I_{n+1}^{\infty} \otimes_{B_r}X) \ar[r]^-{\delta_n} & H_{*-1}(BP_*/I_n^{\infty}\otimes_{B_r}X) \ar[r] & H_{*-1}(BP_*/I_n^{\infty} \otimes_{B_r}X)[v_n^{-1}].}
\]
By \cite[Cor.~8.10]{bhv1}, $\Gamma_{n-1}Y \simeq \Sigma^{-n}BP_*/I_n^{\infty} \otimes Y$ for all $n$ and $Y \in \cD_{BP_*}$. Applying this to $Y = BP_* \otimes_{B_r}X$, we need to show that $\delta_n$ is zero. But $H_{*}(BP_*/I_n^{\infty} \otimes_{B_r}X)$ is $v_n$-torsion free as $X \in \cD_{B_r}$ and $n>r$, hence the second map in the above diagram is injective. 
\end{proof}

\begin{thm}\label{thm:chromaticconvergence}
If $M \simeq q_*N$ for some $N\in \Stable_{W_r}^{< \infty}$, then there is a natural equivalence
\[
\xymatrix{M \ar[r]^-{\sim} & \lim L_{n}M.}
\]
\end{thm}
\begin{proof}
The natural cofiber sequences $\Gamma_{n} \to \Id \to L_{n}$ of functors induce a cofiber sequence
\[
\xymatrix{\lim_n\Gamma_{n}M \ar[r] & M \ar[r] & \lim_nL_{n}M}
\]
for any $M \in \Stable_{BP_*BP}$. Therefore, the claim is equivalent to the statement that $\lim_n\Gamma_{n}M \simeq 0$ whenever $M$ satisfies the assumptions of the theorem. Because $\Gamma_{n}M \in \Stable_{BP_*BP}^{<\infty}$ for all $n\ge 0$, this will follow from \Cref{lem:limzerohomolstructuremaps} once we have shown that the morphism
\[
\xymatrix{H_*\Gamma_{n}(q_*N) \ar[r] & H_*\Gamma_{n-1}(q_*N)}
\]
is zero for all $n>r$. Since $\epsilon_*$ is faithful, it suffices to show that the left vertical map in the following commutative diagram 
\begin{equation}\label{eq:threesquares}
\xymatrix{
\epsilon_*H_*\Gamma_n(q_*N) \ar[r]^-{\sim} \ar[d] & H_*\epsilon_*\Gamma_n(q_*N) \ar[r]^-{\sim} \ar[d] & H_*\Gamma_n^{BP_*}\epsilon_*(q_*N) \ar[r]^-{\sim} \ar[d] & 
H_*\Gamma_{n}^{BP_*}BP_*\otimes_{B_r}\epsilon_*N \ar[d] \\
\epsilon_*H_*\Gamma_{n-1}(q_*N) \ar[r]_-{\sim} & H_*\epsilon_*\Gamma_{n-1}(q_*N) \ar[r]_-{\sim} & H_*\Gamma_{n-1}^{BP_*}\epsilon_*(q_*N) \ar[r]_-{\sim} & 
H_*\Gamma_{n-1}^{BP_*}BP_*\otimes_{B_r}\epsilon_*N
}
\end{equation}
is zero for $n>r$. The commutativity of the first square is clear, while the second one commutes by \cite[Lem.~5.20]{bhv1}; here, the superscript $BP_*$ in $\Gamma^{BP_*}$ indicates that those local cohomology functors are taken in $\cD_{BP_*}$. Finally, the rightmost square commutes because $\Gamma_n$ is smashing together with the commutative diagram
\[
\xymatrix{\Stable_{W_r} \ar[r]^-{q_*} \ar[d]_{\epsilon_*} & \Stable_{BP_*BP} \ar[d]^{\epsilon_*} \\
\cD_{B_r} \ar[r]_-{BP_* \otimes_{B_r}-} & \cD_{BP_*}.}
\]
It therefore remains to show that the right vertical map in \eqref{eq:threesquares} is an equivalence, which is the content of \Cref{lem:chromaticconvergencezeromaps}. 
\end{proof}

\begin{cor}\label{cor:chromaticconvergence}
If $M \in \Stable_{BP_*BP}$ is compact, then there is a natural equivalence
\[
\xymatrix{M \ar[r]^-{\sim} & \lim L_{n}M.}
\]
\end{cor}
\begin{proof}
By \Cref{prop:budscompact}, every compact object in $\Stable_{BP_*BP}$ is given by $(q_r)_*N$ for some $r$ and $N \in \Stable_{W_r}^{<\infty}$. The result thus follows from \Cref{thm:chromaticconvergence}.
\end{proof}

In fact, the algebraic chromatic convergence theorem, \Cref{thm:chromaticconvergence}, can be generalized to comodules with finite projective $BP_*$-dimension, by reducing the statement to its analog for $BP_*$-modules. This argument is essentially due to Hollander; since it has not appeared in print yet, we sketch the argument. 

As in the proof of the previous theorem, 
let $\Gamma_n^{BP_*}$ and $L_n^{BP_*}$ denote the local cohomology functors on $\cD_{BP_*}$ and write $\lim^{BP_*}$ for the total derived functor of inverse limit in this category. 

\begin{lem}\label{lem:acccover}
Suppose $M\in \cD_{BP_*}$ has finite projective dimension, then $M\simeq \lim_n^{BP_*}L_n^{BP_*}M$. 
\end{lem}
\begin{proof}
To simplify the notation, in this proof only we write $\lim$ for $\lim^{BP_*}$. Without loss of generality, assume that $M$ is represented by a complex of projective $BP_*$-modules concentrated in degrees between $0$ and $-k$ for some $k\ge 0$. By \cite[Lem.~5.33]{bhv1}, 
\[
\Gamma_n^{BP_*}M \simeq \Sigma^{-n}M \otimes BP_*/I_{n+1}^{\infty}
\]
is then concentrated in degrees between $n$ and $n-k$. Consequently, $H^s(\Gamma_n^{BP_*}M) = 0$ for all $s<n-k$, i.e., whenever $n>s+k$. The Milnor sequence
\[
\xymatrix{0 \ar[r] & \lim_n^1H^{s-1}(\Gamma_n^{BP_*}M) \ar[r] & H^s(\lim_n\Gamma_n^{BP_*}M) \ar[r] & \lim_n^0 H^{s}(\Gamma_n^{BP_*}M) \ar[r] & 0}
\]
thus implies $\lim_n\Gamma_n^{BP_*}M \simeq 0$ and the claim follows from the usual fiber sequence relating $\Gamma_n^{BP_*}$ and $L_n^{BP_*}$.
\end{proof}

\begin{thm}\label{thm:chromaticconvergence2}
If $M\in \Stable_{BP_*BP}$ has finite projective $BP_*$-dimension, then there is a natural equivalence $M \simeq \lim_n L_{n}M$.
\end{thm}
\begin{proof}
Consider the cosimplicial Amitsur complex 
\[
\xymatrix{C^{\bullet}(M) = (BP_*BP \otimes M \ar@<0.5ex>[r] \ar@<-0.5ex>[r] & BP_*BP^{\otimes 2} \otimes M \ar[l] \ar@<1ex>[r] \ar[r] \ar@<-1ex>[r] & \cdots) \ar@<0.5ex>[l] \ar@<-0.5ex>[l] }
\]
of $M$. By \cite[Thm.~4.29]{bhv1} and \cite[Cor.~5.2.4]{hovey_htptheory}, the canonical map $M \to \Tot(C^{\bullet}(M))$ is an equivalence in $\Stable_{BP_*BP}$. Note that the Amitsur complex is functorial in $M$ and that $C^s(M) = BP_*BP^{\otimes s+1} \otimes M \simeq (\epsilon^*\epsilon_*)^{s+1}M$, where $(\epsilon_*,\epsilon^*)$ is the forgetful-cofree adjunction between $\Stable_{BP_*BP}$ and $\cD_{BP_*}$. Moreover, if $M$ is of finite projective $BP_*$-dimension, then so is $C^s(M)$ for all $s\ge 0$ as $BP_*BP$ is free over $BP_*$. 

Recall that we denote the total derived limit in $\Stable_{BP_*BP}$ and $\cD_{BP_*}$ by $\lim$ and $\lim^{BP_*}$, respectively. Using the fact that $\epsilon^*$ is a right adjoint as well as \cite[Prop.~5.22]{bhv1}, we obtain a sequence of natural equivalences
\begin{align*}
\mathrm{lim}_nL_nM & \simeq \mathrm{lim}_n\Tot(C^{\bullet}(L_nM)) \\
& \simeq \Tot\mathrm{lim}_n((\epsilon^*\epsilon_*)^{\bullet+1}L_nM) \\
& \simeq \Tot\epsilon^*\mathrm{lim}_n^{BP_*}(\epsilon_*(\epsilon^*\epsilon_*)^{\bullet}L_nM) \\
& \simeq \Tot\epsilon^*\mathrm{lim}_n^{BP_*}L_n(\epsilon_*(\epsilon^*\epsilon_*)^{\bullet}M) \\
& \simeq \Tot\epsilon^*\epsilon_*(\epsilon^*\epsilon_*)^{\bullet}M \\
& \simeq \Tot C^{\bullet}(M) \\
& \simeq M,
\end{align*}
where the fifth equivalence comes from \Cref{lem:acccover}. It is straightforward to verify that the composite of these natural maps are compatible with the canonical map $M \to \lim_{n}M$. 
\end{proof}

\begin{rem}
\Cref{thm:chromaticconvergence2} generalizes the algebraic chromatic convergence theorems of Goerss \cite{goerss_quasi-coherent_2008} and Sitte \cite{sitte2014local}. The generality of the theorem is analogous to the generalized (topological) chromatic convergence theorem of \cite{barthel_cc}. However, the topological chromatic convergence theorem does not follow formally from the algebraic version, due to the potential non-convergence of the corresponding inverse limit spectral sequence.  
\end{rem}
\begin{cor}\label{cor:flatcomplete}
	Suppose that either:
	\begin{enumerate}
		\item $M$ is a bounded below $BP_*$-comodule which is flat as a $BP_*$-module; or,
		\item $X$ is a finite complex.
	\end{enumerate}
Then, with $M = BP_*X$ in (2), there is a natural equivalence $M \simeq \lim_n L_{n}M$. 
\end{cor}
\begin{proof}
	The previous theorem reduces the claim to showing that $M$ has finite projective dimension. This follows from \cite[Thm.~4.5]{yosiuma_pd} in the case of (1), and \cite[Cor.~7]{lan79} in the case of (2). 
\end{proof}

\subsection{Further results}

In this subsection, we prove a vanishing result for local cohomology in $\Stable_{BP_*BP}$ and then deduce a comparison theorem for the $E_2$-terms of the Adams--Novikov and $E$-based Adams spectral sequence. Similar, but inequivalent results were originally proven by Goerss in the setting of quasi-coherent sheaves on the moduli of formal groups $\mathcal{M}_{fg}$.

\begin{prop}\label{prop:localcohomvanishing}
Suppose $N \in \Stable_{W_r}^{\le d}$ for some $d \in \Z$, then for all $s > r-n+d$ we have
\[
H_s(\Gamma_{n-1}BP_*\otimes_{B_r}N) = 0.
\]
\end{prop}
\begin{proof}
As in the proof of \Cref{thm:chromaticconvergence}, this is readily reduced to the analogous statement in $\cD_{BP_*}$, namely
\[
H_s(\Gamma_{n-1}BP_*\otimes_{B_r}X) = 0 
\]
for $X \in \cD_{B_r}$ and $s>r-n+d$. In order to prove this, we distinguish two cases. First, assume that $n \le r$. The hypertor spectral sequence \cite[5.7.9 and Thm.~10.6.3]{weibel_homological} takes the form
\[
E_2^{p,q} \cong \bigoplus_{i+j = q}\Tor_p^{BP_*}(H_i(\Gamma_{n-1}BP_*),H_j(BP_*\otimes_{L_r}X)) \implies H_{p+q}(\Gamma_{n-1}BP_*\otimes_{L_r}X).
\]
Since $H_*(\Gamma_{n-1}BP_*) \cong \Sigma^{-n}BP_*/I_{n}^{\infty}$ has flat dimension $n$, then $E_2^{p,q} \ne 0$ only if $p \le n$ and $q \le -n+d$. Therefore, $H_{p+q}(\Gamma_{n-1}BP_*\otimes_{B_r}X) =0$ if $p+q > d$, so certainly $H_{s}(\Gamma_{n-1}BP_*\otimes_{B_r}X) =0$ for $s>r-n+d \ge d$.

For the second case, let $n > r$, and consider the exact sequence
\[
\xymatrix{H_s(\Gamma_{n-1}BP_* \otimes X) \ar[r] & H_s(\Gamma_{n-1}BP_* \otimes X)[v_n^{-1}] \ar[r] & H_s(\Sigma\Gamma_{n}BP_* \otimes X) \ar[r]^-{\delta_n} & \ldots.}
\]
By \Cref{lem:chromaticconvergencezeromaps}, $\delta_n = 0$. Inductively, we know that $H_s(\Gamma_{n-1}BP_* \otimes X) = 0$ if $s > r-n+d$, so it follows that $H_s(\Sigma\Gamma_{n}BP_* \otimes X) = 0$ in the same range. In other words,
\[
H_{s}(\Gamma_{n}BP_* \otimes X) =0
\]
for $s>r - (n+1) +d$.
\end{proof}

\begin{prop}\label{prop:extcomparison}
Let $E=E_n$ be Morava $E$-theory of height $n$. If $M \in \Stable_{BP_*BP}$ satisfies one of the following two conditions:
\begin{enumerate}
	\item there exists $N \in \Stable_{W_r}^{\le d}$ such that $M \simeq q_*N$, or
	\item $M \in \Comod_{BP_*BP}$ (so $d=0$) is of finite projective dimension at most $r-1$,
\end{enumerate}
then the natural localization morphism
\[
\xymatrix{\Ext_{BP_*BP}^s(BP_*,M) \ar[r]^-{l^s} & \Ext_{E_*E}^s(E_*,  E_* \otimes_{BP_*}M)}
\]
is an isomorphism for $s< n-r-d$ and injective for $s = n-r-d$. 
\end{prop}
\begin{proof}
Throughout this proof, we will write $\Ext_{\Psi}(-)$ for the derived primitives $\Ext_{\Psi}(A,-)$ of a Hopf algebroid $(A,\Psi)$. We also use the notation of \Cref{sec:localcohom}; in particular, $(\Phi_*,\Phi^*)$ denotes the base-change adjunction corresponding to $BP_* \to E_*$.

 The morphism $l^s$ is part of an exact sequence 
\begin{equation}\label{eq:extcomparisonles}
\xymatrix{\Ext_{BP_*BP}^s(\Gamma_{n}M) \ar[r] & \Ext_{BP_*BP}^s(M) \ar[r] & \Ext_{BP_*BP}^s(L_{n}M),}
\end{equation}
which is induced by the cofiber sequence $\Gamma_{n}M \to M \to L_{n}M$. Indeed, since $L_{n}M = \Phi^*\Phi_*M$ by \Cref{thm:bhvbpresults}, the last term can be rewritten as
\[
\Ext_{BP_*BP}^s(L_{n}M) \cong \Ext_{E_*E}^s(E_*\otimes_{BP_*}M),
\]
so that the second map in \eqref{eq:extcomparisonles} can be identified with $l^s$. By \Cref{prop:localcohomvanishing}, Condition (1) on $M$ implies that $H^q\Gamma_{n}M = H_{-q}\Gamma_{n}M = 0$ if $-q > r -(n+1) +d$, i.e., for $q \le n-r-d$. If $M$ satisfies the second condition instead, then the same argument as in the proof of \Cref{lem:acccover} shows that $H^q\Gamma_{n}M = 0$ if $q < n-(r-1)$, i.e., $q \le n-r$. Plugging these computations into the hyperext spectral sequence
\[
\Ext_{BP_*BP}^p(H^q\Gamma_{n}M) \implies \Ext_{BP_*BP}^{p+q}(\Gamma_{n}M),
\]
we see that $\Ext_{BP_*BP}^s(\Gamma_{n}M) = 0$ for $s \le n - r - d$, so the claim follows.  
\end{proof}

\begin{rem}
For discrete comodules, the first condition in \Cref{prop:extcomparison} is weaker than the second one, in the following sense: Suppose $M = q_*N$ for some $N \in \Comod_{W_r}$, so $d=0$. By \Cref{lem:budstorsiondim}, $M$ has projective dimension at most $r+1$, so Condition (2) gives an isomorphism $l_s$ for all $s< n-r-2$, while appealing to Condition (1) gives it for $s<n-r$. 
\end{rem}

As an immediate consequence, we obtain:

\begin{cor}\label{cor:comparision}
If $X$ is a $p$-local bounded below spectrum such that $BP_*X$ has projective $BP_*$-dimension $\pdim(BP_*X) \le r$, then the natural map
\[
\xymatrix{\Ext_{BP_*BP}^s(BP_*,BP_*(X)) \ar[r] & \Ext_{E_*E}^s(E_*,E_*(X))}
\]
is an isomorphism if $s < n-r-1$ and injective for $s = n-r-1$. 
\end{cor}

\section{The chromatic spectral sequence}\label{sec:css}

The chromatic spectral sequence was introduced by Miller, Ravenel, and Wilson \cite{mrw_77} as a tool for computing and organizing the $E_2$-term of the Adams--Novikov spectral for the sphere. Splicing together short exact sequences gives the chromatic resolution
\[
\xymatrix@=1em{BP_* \ar[r] & p^{-1}BP_* \ar[rr] \ar[rd] & & v_{1}^{-1}BP_*/p^{\infty} \ar[rr] \ar[rd] & & \ldots \\
& & BP_*/p^{\infty} \ar[ru] & & BP_*/(p^{\infty},v_1^{\infty}) \ar[ru] & & \ldots }
\]
and the resulting spectral sequence is the chromatic spectral sequence. As remarked for example in \cite{greenbook} and \cite{baker_ha}, one can proceed similarly for any bounded below spectrum $X$ with $BP_*X$ flat. 

In this section, we will provide a different construction of the chromatic spectral sequence which works for an arbitrary object $M \in \Stable^{<\infty}_{BP_*BP}$, hence in particular for the $BP$-homology of \emph{any} spectrum $X \in \Sp$. In the case that $M$ is a bounded below flat comodule concentrated in a single degree, our spectral sequence recovers the classical one. However, our approach has several advantages over the classical one, as we will see shortly.

\subsection{The construction}

We will construct our generalization of the chromatic spectral sequence as the Bousfield--Kan spectral sequence associated to the algebraic chromatic tower \eqref{eq:chromatictower}.

\begin{thm}\label{thm:css}
For any $M,N \in \Stable_{BP_*BP}^{< \infty}$, there is a natural convergent spectral sequence
\[
E_1^{n,s,t} \cong \Ext^{s,t}_{BP_*BP}(M,M_{n}N) \implies \Ext^{s,t}_{BP_*BP}(M,L_{\infty}N). 
\]
Furthermore, if $N$ satisfies the conditions of \Cref{thm:chromaticconvergence} or \Cref{thm:chromaticconvergence2}, then the spectral sequence converges to $\Ext_{BP_*BP}(M,N)$.
\end{thm}
\begin{proof}
Applying the functor $\iHom(M,-)$ to the chromatic tower \eqref{eq:chromatictower} of $N$ yields a tower
\[
\xymatrix{\ldots \ar[r] & \iHom(M,L_2N) \ar[r] & \iHom(M,L_1N) \ar[r] & \iHom(M,L_0N) \\
\ldots & \iHom(M,M_2N) \ar[u] & \iHom(M,M_1N) \ar[u] & \iHom(M,M_0N). \ar[u]}
\]
The Bousfield--Kan spectral sequence associated to this diagram, e.g., in the form constructed by Lurie in \cite[Prop.~1.2.2.14]{ha}, thus takes the form
\[
E_1 = \pi_*\iHom(M,M_nN) \implies \pi_*\iHom(M,\lim_nL_nN).
\]
We claim that both $M_nN$ and $L_\infty N$ are in $\Stable_{BP_*BP}^{< \infty}$.  The first claim follows from the fact $M_nN \simeq M_nBP_* \otimes_{BP_*} N$ and \Cref{cor:mn}. For the second one it suffices by \cite[Cor.~1.2.1.6]{ha} and \Cref{prop:linftycc} to show that $L_nN \in \Stable_{BP_*BP}^{< \infty}$. Using the fiber sequence $\Gamma_nN \to N \to L_nN$ we can in turn reduce to showing that $\Gamma_nN \in \Stable_{BP_*BP}^{<\infty}$, which follows from the formula $\Gamma_nN \simeq \Sigma^{-n}BP_*/I_n^\infty \otimes_{BP_*} N$, see \cite[Prop.~8.9]{bhv1}. 

Then, using \Cref{lem:iHomhomotopy} and \Cref{prop:linftycc}, we can rewrite this spectral sequence as 
\[
E_1 \cong \Ext_{BP_*BP}(M,M_nN) \implies \Ext_{BP_*BP}(M,L_{\infty}N).
\]
To see the last part of the claim, it remains to note that $L_{\infty}N \simeq \lim_nL_nN \simeq N$ by \Cref{prop:linftycc} and \Cref{thm:chromaticconvergence}.
\end{proof}

\begin{rem}
Presented in this form, it becomes transparent that the chromatic spectral sequence is completely analogous to the Bousfield--Kan spectral sequence associated to the topological chromatic tower in $\Sp$:
\[
\xymatrix{\ldots \ar[r] & L_2 \ar[r] & L_1 \ar[r] & L_0 \\
\ldots & M_2 \ar[u] & M_1 \ar[u] & M_0. \ar[u]}
\]
Evaluated on $X \in \Sp$, this spectral sequence takes the form 
\begin{equation}\label{eq:tcss}
\pi_*M_nX \implies \pi_*\lim_nL_nX,
\end{equation}
where $M_n$ denotes the $n$th monochromatic functor and $L_n$ is Bousfield localization at Johnson--Wilson theory $E(n)$. If $X$ is chromatically complete \cite{barthel_cc}, then the abutment is equivalent to $\pi_*X$. We will refer to this spectral sequence as the topological chromatic spectral sequence.
\end{rem}

When specialized to the $BP$-homology of the sphere, we recover the classical chromatic spectral sequence. 

\begin{cor}
Suppose $M=BP_*$, then the spectral sequence of \Cref{thm:css} takes the form
\[
E_1 = \Ext^{s,t}_{BP_*BP}(BP_*,v_{n}^{-1}BP_*/I_n^{\infty}) \implies \Ext^{s+n,t}_{BP_*BP}(BP_*,BP_*)
\]
\end{cor}
\begin{proof}
By \Cref{cor:mn} there is an equivalence $M_{n}BP_* \simeq \Sigma^{-n}v_n^{-1}BP_*/I_n^\infty$. 
The result then follows from \Cref{thm:css} and \Cref{cor:chromaticconvergence}.
\end{proof}

\begin{rem}
	More generally, since $M_n$ is smashing, for any $BP_*BP$-comodule $N$ there is a spectral sequence of the form
	\[
E_1^{s,t}= \Ext^{s,t}_{BP_*BP}(BP_*,v_n^{-1}BP_*/I_n^\infty \otimes N) \Rightarrow \Ext^{s+n,t}_{BP_*BP}(BP_*,L_\infty N). 
	\]

	Here the tensor product must be considered in the derived sense. Suppose that $X$ is a spectrum such that $N=BP_*X$ is a bounded below flat $BP_*$-module, then the tensor product is automatically derived, and by \Cref{cor:flatcomplete} the spectral sequence abuts to $\Ext_{BP_*BP}^{\ast,\ast}(BP_*,BP_*X)$.  
\end{rem}
Suppose now that $X$ is a spectrum such that $BP_*(M_nX) \cong \Sigma^{-n} v_n^{-1}BP_*/I_n^\infty \otimes BP_*X$; for example, by \cite[Ch.~8]{orangebook} this is true for the sphere, and hence also whenever $BP_*X$ is a flat $BP_*$-module. If additionally $L_\infty BP_*X \simeq BP_*X$ (e.g., if $X$ is a finite complex, or if $BP_*X$ is bounded below and flat), then it follows that there is a commutative diagram of spectral sequences 
\[
\xymatrix{
	\Ext_{BP_*BP}^{s,t}(BP_*,v_n^{-1}BP_*/I_n^\infty \otimes BP_*X) \ar@{=>}[r]^-{\mathrm{CSS}} \ar[d]_{\cong}& \Ext^{s+n,t}_{BP_*BP}(BP_*,BP_*X) \ar@{=>}[dd]^-{\mathrm{ANSS}} \\
	\Ext_{BP_*BP}^{s,t}(BP_*,\Sigma^n BP_*M_nX) \ar@{=>}[d]_-{\mathrm{ANSS}}& \\
	\pi_{t-s-n}M_nX \ar@{=>}[r]_-{\mathrm{TCSS}} & \pi_{t-n-s}X,
}
\]
relating the chromatic spectral sequence (CSS) with the Adams--Novikov spectral sequence (ANSS) and the topological chromatic spectral sequence (TCSS).

\subsection{The finite height chromatic spectral sequence}

It is easy to derive a finite height analog of the chromatic spectral sequence from \Cref{thm:css}. First, we need a base-change lemma. We use the notation of \Cref{sec:localcohom}; in particular, $(\Phi_*,\Phi^*)$ denotes the base-change adjunction corresponding to $BP_* \to E(n)_*$.

\begin{lem}\label{lem:e1termbasechange}
For any $X,Y \in \Stable_{BP_*BP}$, there is a natural equivalence
\[
\Hom_{BP_*BP}(X,M_nY) \simeq \Hom_{E(n)_*E(n)}(E(n) \otimes X, E(n)_*/I_n^{\infty}\otimes Y),
\]
with $E(n)_{*}/I_n^{\infty}\otimes Y \simeq \Gamma_{n-1}^{E(n)_*}\Phi_*Y$.
\end{lem}
\begin{proof}
 There are natural equivalences
\[
M_n = L^{BP_*}_{n}\Gamma^{BP_*}_{n-1} \simeq \Phi^*\Phi_*\Gamma^{BP_*}_{n-1} \simeq \Phi^*\Gamma^{E(n)_*}_{n-1}\Phi_*,
\]
the last one resulting from the equivalence $E(n)_* \otimes{BP_*} BP_*/I_{n}^{\infty} \simeq E(n)_*/I_{n}^{\infty}$. Consequently, by adjunction we obtain
\[
\Hom_{BP_*BP}(X,M_nY) \simeq \Hom_{E(n)_*E(n)}(\Phi_*X,\Gamma^{E(n)_*}_{n-1}\Phi_*Y),
\]
and the claim follows.  
\end{proof}

\begin{prop}
Fix an integer $n\ge 0$. For any $X,Y \in \Stable^{< \infty}_{BP_*BP}$, there is a natural strongly convergent spectral sequence of the form
\[
E_1^{k,s,t} = \begin{cases}
	\Ext^{s,t}_{E(k)_*E(k)}(E(k)_* \otimes X,E(k)_*/I_k^{\infty}\otimes Y) & k \le n \\
	0 & k > n
\end{cases}
\]
converging to $\Ext^{s,t}_{BP_*BP}(X,L_nY) \cong \Ext^{s,t}_{E(n)_*E(n)}(E(n)_* \otimes_{BP_*} X,E(n)_* \otimes_{BP_*}Y)$.
\end{prop}
\begin{proof}
Truncating the chromatic tower at height $n$ and using the same argument as in the proof of \Cref{thm:css}, we obtain a strongly convergent spectral sequence
\[
E_1^{k\le n} = \Ext_{BP_*BP}(X,M_{k}Y) \implies \Ext_{BP_*BP}(X,L_{n}Y). 
\] 
Applying \Cref{lem:e1termbasechange} to this, we obtain the desired $E_1$-term. The identification of the abutment follows a similar argument. 
\end{proof}
We thus recover \cite[Thm.~5.1]{hov_sadofsky}:

\begin{cor}
	The chromatic spectral sequence converging to $\Ext_{E(n)_*E(n)}^{s,t}(E(n)_*,E(n)_*)$ has $E_1$-term 
	\[
E_1^{k,s,t} = \begin{cases}
	\Ext^{s,t}_{E(k)_*E(k)}(E(k)_*,E(k)_*/I_k^{\infty}) & k \le n \\
	0 & k > n.
\end{cases}
\]
\end{cor}
This spectral sequence was used by Hovey and Sadofsky in their calculations of the $E(n)$-local Picard group, see  \cite{hov_sadofsky}.

\bibliography{duality}
\bibliographystyle{alpha}
\end{document}